\documentclass[a4paper,oneside,11pt, reqno, psamsfonts]{amsart}
\usepackage{graphicx}
\usepackage{float}
\usepackage{pstricks}
\usepackage{amscd}
\usepackage{amsmath}
\usepackage{amsxtra}
\usepackage[T1]{fontenc}
\usepackage[utf8]{inputenc}
\usepackage{lipsum}
\usepackage{tikz}
\usepackage{amsfonts,amssymb,amsthm}
\usepackage{url}
\usepackage{enumerate}
\usetikzlibrary{arrows}
\usetikzlibrary{calc}
\usepackage{url}
\newcommand{\strutstretchdef}{\newcommand{\strutstretch}{\vphantom{\raisebox{1pt}{$\big($}\raisebox{-1pt}{$\big($}}}}
\theoremstyle{plain}
\newtheorem{theorem}{Theorem}[section]

\newtheorem{lemma}[theorem]{Lemma}
\newtheorem{proposition}[theorem]{Proposition}
\newtheorem{corollary}[theorem]{Corollary}

\theoremstyle{definition}

\newtheorem{example}[theorem]{Example}
\newtheorem{question}[theorem]{Question}

\theoremstyle{remark}
\newtheorem{remark}[theorem]{Remark}

\numberwithin{equation}{section}

\newlength{\struh}
\newlength{\textminustop}
\strutstretchdef
\usepackage{mathrsfs}
\usepackage{epsfig}
\usepackage[active]{srcltx}

\newcommand{\ncom}{\newcommand}
\ncom{\bq}{\begin{equation}}
\ncom{\eq}{\end{equation}}
\ncom{\beqn}{\begin{eqnarray*}}
\ncom{\eeqn}{\end{eqnarray*}}
\ncom{\beq}{\begin{eqnarray}}
\ncom{\eeq}{\end{eqnarray}}
\ncom{\nno}{\nonumber}
\ncom{\rar}{\rightarrow}
\ncom{\Rar}{\Rightarrow}
\ncom{\noin}{\noindent}
\ncom{\bc}{\begin{centre}}
\ncom{\ec}{\end{centre}}
\ncom{\sz}{\scriptsize}
\ncom{\rf}{\ref}
\ncom{\sgm}{\sigma}
\ncom{\Sgm}{\Sigma}
\ncom{\dt}{\delta}
\ncom{\Dt}{Delta}
\ncom{\lmd}{\lambda}
\ncom{\Lmd}{\Lambda}
\ncom{\eps}{\epsilon}
\ncom{\pcc}{\stackrel{P}{>}}
\ncom{\dist}{{\rm\,dist}}
\ncom{\sspan}{{\rm\,span}}
\ncom{\im}{{\rm Im\,}}
\ncom{\sgn}{{\rm sgn\,}}
\ncom{\ba}{\begin{array}}
\ncom{\ea}{\end{array}}
\ncom{\eop}{\hfill{{\rule{2.5mm}{2.5mm}}}}
\ncom{\eoe}{\hfill{{\rule{1.5mm}{1.5mm}}}}
\ncom{\eof}{\hfill{{\rule{1.5mm}{1.5mm}}}}
\ncom{\hone}{\mbox{\hspace{1em}}}
\ncom{\htwo}{\mbox{\hspace{2em}}}
\ncom{\hthree}{\mbox{\hspace{3em}}}
\ncom{\hfour}{\mbox{\hspace{4em}}}
\ncom{\hsev}{\mbox{\hspace{7em}}}
\ncom{\vone}{\vskip 2ex}
\ncom{\vtwo}{\vskip 4ex}
\ncom{\vonee}{\vskip 1.5ex}
\ncom{\vthree}{\vskip 6ex}
\ncom{\vfour}{\vspace*{8ex}}
\ncom{\norm}{\|\;\;\|}
\ncom{\integ}[4]{\int_{#1}^{#2}\,{#3}\,d{#4}}
\ncom{\inp}[2]{\langle{#1},\,{#2} \rangle}
\ncom{\Inp}[2]{\Langle{#1},\,{#2} \Langle}
\ncom{\vspan}[1]{{{\rm\,span}\#1 \}}}
\ncom{\dm}[1]{\displaystyle {#1}}

\begin{document}
\title[Dirichlet-type spaces and joint $2$-isometries]{Dirichlet-type spaces on the unit ball \\ and joint $2$-isometries}

%\title{Towards a representation theorem for analytic joint $m$-isometries}

\author[S. Chavan, R. Gupta and Md. R. Reza]{Sameer Chavan, Rajeev Gupta and Md. Ramiz Reza}

%\thanks{The second named author was partially supported by NSF Grant DMS-1302666.} \

\address{Department of Mathematics and Statistics\\
Indian Institute of Technology Kanpur, India}
   \email{chavan@iitk.ac.in}
   \email{rajeevg@iitk.ac.in}
 \email{ramiz@iitk.ac.in}

\thanks{The work of the second author is supported through the Inspire Faculty Fellowship}

\keywords{Poisson integral, Dirichlet-type spaces, Carleson measure, complex moment problem, joint $m$-isometry}

\subjclass[2000]{Primary 47A13, 47B38; Secondary 31C25, 46E20, 44A60}

\begin{abstract} 
We obtain a formula that relates the spherical moments of the multiplication tuple  on a Dirichlet-type space  to a complex moment problem in several variables.  This can be seen as the ball-analogue of a formula originally invented by Richter in \cite{Ri}.
%In case of dimension $1,$ this was derived by Richter in \cite{Ri}.
We capitalize on this formula to study  Dirichlet-type spaces on the unit ball and joint $2$-isometries. 
\end{abstract}

\setcounter{tocdepth}{2}

\maketitle

%\tableofcontents

\vspace{-.8cm}

\section{Introduction}

The investigations in this paper are motivated by the theory of Dirichlet-type spaces first introduced and studied by Richter \cite{Ri} in the context of the wandering subspace problem and the model theory of cyclic analytic $2$-isometries. 
%(see \cite{Al} for a generalization to the case of cyclic complete hyperexpansions). 
The present work may be seen as an attempt to understand the spherical counterpart of the theories of Dirichlet-type spaces and $2$-isometries.  
 It is well-known that every analytic $2$-isometry admits the wandering subspace property (see \cite[Theorem 1]{Ri-2} and \cite[Theorem 3.6]{Sh}). 
 In case $d > 1$, there are analytic joint $2$-isometric $d$-tuples without the wandering subspace property. Indeed, by \cite[Example 6.8]{BEKS}, for every non-zero $a$ in the open unit ball $\mathbb B^d$ in $\mathbb C^d,$ the restriction of the Drury-Arveson $d$-shift to the invariant subspace $\mathcal M_a:=\{f \in H^2_d : f(a)=0\}$ admits the wandering subspace property if and only if $d=1,$ and one may infer from \cite[Theorem 4.2]{GR} that the restriction of the Drury-Arveson $d$-shift to $\mathcal M_a$  is a joint $d$-isometry.  Further, in the context of model theory for joint $2$-isometric $d$-tuples, there is a disparity in the cases $d=1$ and $d >1.$ In fact, by \cite[Theorem 5.1]{Ri}, the operator of the multiplication by the coordinate function on a Dirichlet-type space provides a canonical model for any cyclic analytic $2$-isometry. In case of $d > 1,$ there exist cyclic analytic joint $2$-isometric $d$-tuples which can not be modeled in this way (see Example \ref{a-model-fails}; see also \cite[Theorem 3.2]{Le}). 
 
We also find it necessary to comment on possible variants of the Dirichlet-type spaces in several variables. It is well-known \cite{Ru-2} that there are two possible analogues of Poisson kernel for the open unit ball $\mathbb B^d$ in $\mathbb C^d$ provided $d \geqslant 2.$  One of which is the {\it invariant Poisson kernel} $P_{\iota}(z, \zeta)$ given by
\beqn
P_{\iota}(z, \zeta) = \frac{(1-\|z\|^2)^d}{|1- \inp{z}{\zeta}|^{2d}}, \quad z \in \mathbb B^d, ~ \zeta \in \partial \mathbb B^d.
\eeqn
One may associate with any finite positive Borel measure $\mu$ on the unit sphere $\partial \mathbb B^d$ the {\it invariant Poisson integral} $$P_{\iota}[\mu](z)=\int_{\partial \mathbb B^d} P(z, \zeta) d\mu(\zeta), \quad z \in \mathbb B^d,$$ and define the {\it invariant Dirichlet-type space} $\mathscr D_{\iota}(\mu)$ as the Hilbert space of those functions $f$ in the Hardy space $H^2(\mathbb B^d)$ for which the weighted $L^2$-norm of the gradient of $f,$ with weight being $P_{\iota}[\mu]$ is finite. 
Unlike the case of $d=1,$ the invariant Poisson integral need not define a harmonic function (see \cite[Remark 3.3.10]{Ru-2}).
In particular, the tuple $\mathscr M_z$ of the multiplication by the coordinate functions on an invariant Dirichlet-type space need not be a joint $2$-isometry (see Remark \ref{R-F-fails}). Thus the verbatim analogue of \cite[Theorem 3.7]{Ri-2} fails for invariant Dirichlet-type spaces.
%On the other hand, 
%the technique employed in this paper relies on the solution of the classical Dirichlet problem, and hence it is not applicable in the study of invariant Dirichlet-type spaces unless the  invariant Poisson integral coincides with the Euclidean Poisson integral as discussed below. 
Rather unexpectedly, the {\it Poisson kernel} $P(z, \zeta)$ associated with the Euclidean ball of $\mathbb R^{2d}$ yields Dirichlet-type spaces that support joint $2$-isometries: 
%\beqn
% P(z, \zeta) = \frac{1-\|z\|^2}{\|z-\zeta\|^{2d}}, \quad z \in \mathbb B^d, ~ \zeta \in \partial \mathbb B^d.
% \eeqn
% The {\it Poisson kernel} for $\mathbb B^d$ (considered as a subset of $\mathbb R^{2d}$) is given by 
 \beqn 
 %\label{Euclid-p}
 P(z, \zeta) = \frac{1-\|z\|^2}{\|z-\zeta\|^{2d}}, \quad z \in \mathbb B^d, ~ \zeta \in \partial \mathbb B^d.
 \eeqn
 Throughout this article, we need the following properties of the Poisson kernel (see \cite[Proposition 3.1.12]{Si}): 
\begin{enumerate}
\item[$\bullet$] For every $\zeta \in \partial \mathbb B^d,$ \beq \label{P-nor} \int_{\mathbb B^d} P(z, \zeta)dV(z)=1, \eeq
where $V$ is the normalized volume measure on $\mathbb B^d.$
\item[$\bullet$] For every $z \in \mathbb B^d,$ \beq \label{P-nor-2} \int_{\partial \mathbb B^d} P(z, \zeta)d\sigma(\zeta)=1, \eeq
where $\sigma$ is the normalized surface area measure on $\partial \mathbb B^d.$
\item[$\bullet$] $P(z,\zeta)$ is symmetric in the following sense:
 \beq \label{P-sym} P(r\eta, \zeta)=P(r\zeta, \eta), \quad \eta, \zeta \in \partial \mathbb B^d, ~0 \leqslant r < 1.
 \eeq 
%\item[$\bullet$] For every $\zeta \in \partial \mathbb B^d,$ \beq 
%\label{P-harmonic} P(\cdot,\zeta)~\mbox{is harmonic on}~ \mathbb B^d.
%\eeq
\item[$\bullet$] For any neighbourhood $N_{\zeta}$ of $\zeta \in \partial \mathbb B^d,$ 
\beq \label{pres-nbhd} \lim_{r \rar 1^{-}} P(r\eta, \zeta) =0, \quad \eta \in \partial \mathbb B^d \setminus N_{\zeta}. \eeq
\end{enumerate} 
 It is worth noting that  \eqref{P-nor} may be deduced from  \eqref{P-nor-2} and \eqref{P-sym} in view the polar coordinates (see \cite[Pg 13]{Ru-2}).
%\beq
%\label{pc} \int_{\mathbb B^d} f(z)dV(z) = 2d \int_{0}^1 r^{2d-1} \Big(\int_{\partial \mathbb B^d} f(r \zeta)d\sigma(\zeta)\Big)dr, \quad f \in C(\mathbb B^d),
%\eeq
%where $C(X)$ denotes the vector space of complex-valued continuous functions on a topological space $X.$
The analysis of this paper relies heavily on the existence and uniqueness of the solution of the Dirichlet problem for the unit ball (see \cite[Theorem 3.1.13]{Si}). Recall that $C(X)$ denotes the vector space of complex-valued continuous functions on a topological space $X.$
\begin{theorem} \label{DP}
For every $f \in C(\partial \mathbb B^d),$ there is a unique $u \in C(\overline{\mathbb B}^d),$ so that $u$ is harmonic on $\mathbb B^d$ and $u|_{\partial \mathbb B^d}=f.$ Moreover,
\beqn 
u(z)=\int_{\partial \mathbb B^d} P(z, \zeta)f(\zeta)d\sigma(\zeta), \quad z \in \mathbb B^d.
\eeqn
\end{theorem}
As in the case of invariant Poisson kernel, one may associate with a finite positive Borel measure $\mu$ on $\partial \mathbb B^d$  the {\it Poisson integral} 
\beqn  P[\mu](z)=\int_{\partial \mathbb B^d} P(z, \zeta) d\mu(\zeta), \quad z \in \mathbb B^d, \eeqn 
and form the Dirichlet-type space $\mathscr D(\mu)$ (see Section 3 for a precise definition in a more general setting). 
It turns out that $P[\mu]$ can be defined for any complex Borel measure $\mu$, and in this case, $P[\mu]~\mbox{is~a~harmonic~function~on}~\mathbb B^d.$
The Riesz-Herglotz Theorem asserts that every positive harmonic function on the unit ball $\mathbb B^d$ is the Poisson integral of a finite positive Borel measure on $\partial \mathbb B^d$ (see, for instance, \cite[Corollary 6.15]{ABR}). 

The investigations in this paper are motivated by the following questions:

\begin{question}  \label{Q1.1}
Assume that $d$ is a positive integer and let $\mu$ be a finite positive Borel measure on $\partial \mathbb B^d.$ 
\begin{enumerate}
\item[(a)] Is $z_j,$ $j=1, \ldots, d$ a multiplier of the Dirichlet-type space $\mathscr D(\mu)$? 
\item[(b)] If $z_j,$ $j=1, \ldots, d$ is a multiplier of $\mathscr D(\mu),$ is the $d$-tuple $\mathscr M_z$ of multiplication operators $\mathscr M_{z_1}, \ldots, \mathscr M_{z_d}$ a joint $2$-isometry?
\end{enumerate}
\end{question}
It is worth noting that in dimension $d=1$, the answers to the above questions are affirmative \cite[Theorems 3.6 and 3.7]{Ri}. Moreover, these answers are intimately related to the following formula obtained in \cite[Proof of Theorem 4.1]{Ri} in case of $k=1$ (see Lemma \ref{Richter-char} for the deduction of the general case from this one).
\begin{theorem}[Richter's formula] 
%\label{Richter-one-var}
For any finite positive Borel measure $\mu$ on the unit circle $\partial \mathbb D,$  we have
\beqn
 && \int_{\mathbb D}  (z^{k}p(z))' \,\overline{(z^{k}q(z))'} \, P[\mu](z) dA(z)  \\&=& \int_{\mathbb D}  p'(z)\,\overline{q'(z)} \, P[\mu](z) dA(z) + k \int_{\partial \mathbb D} p(\zeta)\, \overline{q(\zeta)} d\mu(\zeta), \!\!\! \quad p, q \in \mathbb C[z], ~ k \geqslant 1,
\eeqn
where $dA$ denotes the normalized area measure on the unit disc $\mathbb D,$ $\mathbb C[z]$ is the complex vector space of polynomials in $z,$ and $f'$ denotes the complex derivative of $f \in \mathbb C[z].$ 
\end{theorem}
\begin{remark} 
By the formula above,  $\|z^kp\|^2_{\mathscr D(\mu)}$ is a linear polynomial in $k$ for every $p \in \mathbb C[z].$  
Further, if $p$ is a polynomial and $p(z)=p(0) + zq(z)$ for some polynomial $q,$ then by a couple of applications of Richter's formula, 
\beqn
&& \frac{1}{2}\Bigg(\int_{\mathbb D}  |(zp(z))'|^2  \, P[\mu](z) dA(z) -\int_{\mathbb D}  |p'(z)|^2  \, P[\mu](z) dA(z) \Bigg)\\
 %&=&   \int_{\partial \mathbb D} | p(\zeta)|^2 d\mu(\zeta)  \\ 
& \leqslant &   \int_{\partial \mathbb D} | q(\zeta)|^2 d\mu(\zeta)  
+ |p(0)|^2 \mu(\partial \mathbb D)\\
& = & \int_{\mathbb D}  |(zq(z))'|^2  \, P[\mu](z) dA(z) -\int_{\mathbb D}  |q'(z)|^2  \, P[\mu](z) dA(z) + |p(0)|^2 \mu(\partial \mathbb D) \\
& \leqslant & \int_{\mathbb D}  |p'(z)|^2  \, P[\mu](z) dA(z) + \|p\|^2_{H^2(\mathbb D)} \mu(\partial \mathbb D).
\eeqn
Since 
the polynomials are dense in $\mathscr D(\mu)$ (see \cite[Corollary 3.8(d)]{Ri}), $z$ is a multiplier for $\mathscr D(\mu)$ and  the operator $\mathscr M_z$ of the multiplication by coordinate function $z$ is a $2$-isometry.
This recovers \cite[Theorems 3.6 and 3.7]{Ri}.
 %This fact can be derived without the polynomial density in $\mathscr D(\mu).$
\end{remark}
One of the main results of this paper is the ball analogue of Richter's formula (see Theorem \ref{Richter-several}). Our method of proof, that exploits the Dirichlet problem for the unit ball, Green's Theorem and an approximation result, differs from the one employed in \cite{Ri}.
This allows us to answer Question \ref{Q1.1}(b) in the affirmative (see Theorem \ref{R-F-role-1}). We also answer Question \ref{Q1.1}(a) for weighted surface area measures with bounded measurable weight functions (see Corollary \ref{Coro-multi}). 
In the remaining half of this paper, we discuss the ball-analogue of the trigonometric moment problem  (see Theorem \ref{CMP} and Appendix), and employ it to characterize joint $m$-isometries that admit the wandering subspace property (see Theorem \ref{main-thm}). These results are then combined with the theory of Dirichlet-type spaces developed in the first half to model spherical moments of joint $2$-isometries (see Corollary \ref{a-model}).
%We conclude the paper with some open problems.

%\beqn
%\|z_j\|^2 = \int_{\mathbb B^d} \tilde{P}(z, e_1)dV(z) = 2d \int_{0}^1 r^{2d-1}\int_{\partial \mathbb B^d} \tilde{P}(r\eta, e_1)d\sigma(\eta)dr \\= 2d \int_{0}^1 r^{2d-1}\int_{\partial \mathbb B^d} \tilde{P}(re_1, \eta)d\sigma(\eta)dr =1
%\eeqn
%
%\beqn
%\sum_{i, j=1}^d \|z_jz_i\|^2 = 2(d+1)\int_{\mathbb B^d} \|z\|^2 \tilde{P}(z, e_1)dV(z) = 2(d+1)2d \int_{0}^1 r^{2d+1}\int_{\partial \mathbb B^d} \tilde{P}(r\eta, e_1)d\sigma(\eta)dr \\= 2(d+1)2d \int_{0}^1 r^{2d+1}\int_{\partial \mathbb B^d} P(re_1, \eta)d\sigma(\eta)dr =2(d+1)\frac{d}{d+1}=2d
%\eeqn
%
%\beqn
%dvd
%\eeqn

%\section{The cone $\mathscr{R\!\!M}_{\!\!+\!}(\partial \mathbb B^d)$}
\section{Richter's formula in several variables}

Before we state the main result of this section, let us set some standard notations.
Let $\mathbb Z_+$ denote the set of non-negative integers. For a positive integer $d,$ the
$d$-fold Cartesian product of $\mathbb Z_+$ is denoted by $\mathbb Z_+^d.$ For $p =(p_1, \ldots,
p_d) \in \mathbb Z_+^d,$ set $|p|=p_1 + \cdots + p_d.$  If $q = (q_1, \ldots, q_d) \in  \mathbb Z_+^d,$ then we set $q!=\prod_{j=1}^d q_j!.$ We write
$p \leqslant q$ if $p_j \leqslant q_j$ for every $j=1, \ldots, d.$ 
%and in this case, $q \choose p$ is understood to be the product $\prod_{j=1}^d{q_j \choose p_j}.$ 
Let $\mathbb C$ denote the set of complex numbers and let $\mathbb C^d$ denote the $d$-fold Cartesian product of $\mathbb C$. We reserve the notations $\Re(z), \overline{z}$ and $|z|$ for the real part, the complex conjugate and the modulus of the complex number $z,$ respectively. Let $\mathbb T$ denote the unit circle in $\mathbb C$ and let $\mathbb T^d$ be the $d$-fold Cartesian product of $\mathbb T.$ For $z=(z_1, \ldots, z_d), w=(w_1, \ldots, w_d) \in \mathbb C^d,$ let $\inp{z}{w}=\sum_{j=1}^dz_j \overline{w}_j$ and $\|z\|=\inp{z}{z}^{1/2}.$ For $j=1, \ldots, d,$ let $\varepsilon_j$ denote the $d$-tuple in $\mathbb C^d$ with $1$ in the $j$-th entry and zeros elsewhere. For $z \in \mathbb C^d$ and $\alpha \in \mathbb Z^d_+,$ let $z^{\alpha}$ denote the complex number $\prod_{j=1}^dz^{\alpha_j}_j.$
For $R > 0,$ $R\mathbb B^d$ denotes  the open ball in $\mathbb C^d$ centered at the origin and of radius $R.$ In case $R=1,$ we denote $R\mathbb B^d$ simply by $\mathbb B^d.$ The topological boundary and closure 
of $\mathbb B^d$ are denoted by $\partial \mathbb B^d$ and $\overline{\mathbb B}^d,$  respectively. 
The vector space of complex polynomials in $z_1, \ldots, z_d$ is denoted by $\mathbb C[z_1, \ldots, z_d].$ By abuse of notation, the space $\mathbb C[z_1, \ldots, z_d]|_{\Omega}$ of restriction of complex polynomials to a subset $\Omega$ of $\mathbb C^d$ is also denoted by $\mathbb C[z_1, \ldots, z_d].$ For a set $A,$ let $\ell^2(A)$ denote the Hilbert space of complex-valued functions on $A,$ which are square integrable with respect to the counting measure. 
Let $M_+(\partial \mathbb B^d)$ denote the cone of finite positive Borel measures on the unit sphere $\partial \mathbb B^d$ in $\mathbb C^d.$
Let $\nabla$ denote the gradient $\big(\frac{\partial }{\partial z_1}, \ldots, \frac{\partial }{\partial z_d}\big)$ and $\Delta$ denotes the complex Laplacian $\sum_{j=1}^d \frac{\partial^2 }{\partial \overline{z_j}\partial z_j}$ in dimension $d.$

We now state the main result of this section.
\begin{theorem} \label{Richter-several}
Let $\mu \in M_+(\partial \mathbb B^d).$  Then
\beq \label{formula-f}
&& \sum_{|\gamma|=k} \frac{|\gamma|!}{\gamma!}  \int_{\mathbb B^d} \inp{\nabla z^{\gamma}p(z)}{\nabla z^{\gamma}q(z)} P[\mu](z) dV(z)  \\ &=&  \int_{\mathbb B^d} \inp{\nabla p(z)}{\nabla q(z)} P[\mu](z) dV(z) ~+~ kd \int_{\partial \mathbb B^d} p(\zeta) \overline{q(\zeta)} d\mu(\zeta), \notag \\ 
 && \notag \hspace{5cm} p, q \in \mathbb C[z_1, \ldots, z_d], ~ k \in  \mathbb Z_+.
\eeq
\end{theorem} 

The proof of Theorem \ref{Richter-several}, as presented below, consists of several lemmas.
%We begin with a simple characterization of the class $\mathscr{R\!\!M}_{\!\!+\!}(\partial \mathbb B^d)$.

\begin{lemma} \label{Richter-char}
Let $\mu \in M_+(\partial \mathbb B^d).$  Then 
\eqref{formula-f} holds if and only if \eqref{formula-f} holds for $k=1,$ that is,
\beq
\label{Richter's measure}
&&  \sum_{j=1}^d   \int_{\mathbb B^d} \inp{\nabla z^{\alpha + \varepsilon_j}}{\nabla z^{\beta + \varepsilon_j}} P[\mu](z) dV(z)  \\
&=&   \int_{\mathbb B^d} \inp{\nabla z^{\alpha}}{\nabla z^{\beta}} P[\mu](z) dV(z) + d \int_{\partial \mathbb B^d} \zeta^{\alpha} \overline{\zeta}^{\beta} d\mu(\zeta), \quad \alpha, \beta \in \mathbb Z^d_+. \notag
\eeq
\end{lemma}
\begin{proof} The necessity part is clear. To see the sufficiency part, note that by the linearity of the integral, \eqref{Richter's measure} implies \eqref{formula-f} in case of $k=1.$ To see the general case, for fixed $\alpha, \beta \in \mathbb Z^d_+,$ let
\beqn
Q_k(z) = \sum_{|\gamma|=k} \frac{|\gamma|!}{\gamma!}  \inp{\nabla z^{\alpha + \gamma}}{\nabla z^{\beta + \gamma}}, \quad z \in \mathbb B^d, ~k \in \mathbb Z_+.
\eeqn 
Let $k$ be a positive integer and note that $Q_k -Q_0 = \sum_{j=1}^k (Q_j - Q_{j-1})$ 
and 
\beqn Q_j(z) &=& \sum_{i_1, \ldots, i_j=1}^d \inp{\nabla  z^{\alpha +  \varepsilon_{i_1} + \cdots + \varepsilon_{i_j}}}{\nabla z^{\beta +  \varepsilon_{i_1} + \cdots + \varepsilon_{i_j}}} \\ &=&
\sum_{l=1}^d \sum_{|\gamma|=j-1} \frac{|\gamma|!}{\gamma!}     \inp{\nabla  z^{\alpha + \gamma + \varepsilon_l}}{\nabla z^{\beta + \gamma + \varepsilon_l}}, \quad z \in \mathbb B^d, ~j=1, \ldots, k.
\eeqn
Thus, by \eqref{Richter's measure} and an application of the multinomial theorem,  
\beqn
&& \int_{\mathbb B^d} \big(Q_k(z) -Q_{0}(z)\big)P[\mu](z)dV(z) \\ &=&\int_{\mathbb B^d} \sum_{j=1}^k \big(Q_j(z) -Q_{j-1}(z)\big)P[\mu](z)dV(z) \\ &=& d \sum_{j=1}^k \sum_{|\gamma|=j-1} \frac{|\gamma|!}{\gamma!}   \int_{\partial \mathbb B^d} \zeta^{\alpha + \gamma} \overline{\zeta}^{\beta + \gamma} d\mu(\zeta) \\ &=& d \sum_{j=1}^k \int_{\partial \mathbb B^d} \Big(\sum_{|\gamma|= j-1} \frac{|\gamma|!}{\gamma!} \zeta^{\gamma} \overline{\zeta}^{\gamma}  \Big) \zeta^{\alpha} \overline{\zeta}^{\beta} d\mu(\zeta) \\ &=& kd\int_{\partial \mathbb B^d} \zeta^{\alpha} \overline{\zeta}^{\beta} d\mu(\zeta).
\eeqn
This completes the verification.
%Since $\{\gamma \in \mathbb Z^d_+: |\gamma|=k+1\}=\{\gamma +\varepsilon_j \in \mathbb Z^d_+: |\gamma|=k, ~j=1, \ldots, d\},$ $k \in \mathbb Z_+,$
%the general case may be obtained as in the proof of Theorem \ref{Richter-one-var}.
\end{proof}

%{\red Motivate to Carleson-type condition. Further, by Proposition \ref{cross-grad} and the fact that {\red Richter's lemma-Thm 4.1}, any positive Borel finite measure on the unit circle belongs to the class $\mathscr{R\!\!M}_{\!\!+\!}(\partial \mathbb B^1).$} 

The following identity relates the gradient and the complex Laplacian. 
\begin{lemma} \label{Rajeev} For holomorphic functions $f, g : \mathbb B^d \rar \mathbb C$
 and constant $C,$ $$\sum_{j=1}^d \inp{\nabla z_j f}{\nabla z_jg} - C\inp{\nabla f}{\nabla g} =\Delta ((\|z\|^2-C)f\overline{g}).$$
\end{lemma}
\begin{proof}
For holomorphic functions $f, g : \mathbb B^d \rar \mathbb C,$ note that $\inp{\nabla f}{\nabla g} =  \Delta(f\overline{g}).$ Thus
 \beqn
 \sum_{j=1}^d \inp{\nabla z_jf}{\nabla z_j g} - C\inp{\nabla f}{\nabla g}  &=&   \sum_{j=1}^d \Delta( z_j f \overline{z_j}\,\overline{g}) - C\Delta( f\overline{g}) \\ &=&   \Delta( (\|z\|^2-C)f \overline{g}),
 \eeqn
 where we used the linearity of $\Delta.$
\end{proof}
%\begin{remark}
%Note that
% $\sum_{j=1}^d \inp{\nabla z^{\alpha + \varepsilon_j}}{\nabla z^{\beta + \varepsilon_j}} = \|z\|^2 \inp{\nabla z^{\alpha}}{\nabla z^{\beta}} + (|\alpha|+|\beta|+d)z^{\alpha}\overline{z}^{\beta}.$ 
%\end{remark}

We next state a special case of Green's Theorem required in the proof of Theorem \ref{Richter-several}.
\begin{lemma}  
Let $f,g$ be $C^2$-functions on $\mathbb B^d$ so that $f, g$ and $D^{\alpha}f, D^{\alpha}g$ have continuous extensions to $\overline{\mathbb B}^d$ for all multi-indices, $\alpha$, with $|\alpha| \leqslant 2$. Then
\beq \label{Green} \notag
&& \int_{\mathbb B^d}\Big((\Delta f)(z) g(z) - (\Delta g)(z) f(z) \Big)dV(z) \\ &=&  \frac{d}{2}  \int_{\partial \mathbb B^d}\Big(\frac{\partial f(r\zeta)}{\partial r}\Big|_{r=1} g(\zeta) - \frac{\partial g(r\zeta)}{\partial r}\Big|_{r=1} f(\zeta) \Big)d\sigma(\zeta).
\eeq
%where $\triangle$ denotes the complex Laplacian.
\end{lemma}
\begin{proof}
Recall the well-known fact that the ratio of the volume of $\mathbb B^d$ and the surface area of $\partial \mathbb B^d$ equals $1/2d$ (see \cite[Pg 291]{Si-0}). This combined with the facts that 
  the (Euclidean) Laplacian is $4$ times the complex Laplacian $\Delta$ and the outward pointing derivative is the radial derivative, the desired conclusion may be deduced from \cite[Theorem 1.4.1]{Si} (where non-normalized volume measure and surface area measure have been used).
\end{proof}

The following is the ball analogue of \cite[Lemma 3.4]{Ri}.
\begin{proposition} \label{Richter-char-R}
Let $\mu$ be a complex Borel measure on $\partial \mathbb B^d$ and $0 \leqslant R < 1.$ Then, for holomorphic functions $f, g :\mathbb B^d \rar \mathbb C,$ 
\beqn
%\label{Richter's measure-R}
&& \notag \sum_{j=1}^d   \int_{R\mathbb B^d} \inp{\nabla z_jf}{\nabla z_jg} P[\mu](z) dV(z)  -
 R^2 \int_{R\mathbb B^d} \inp{\nabla f}{\nabla g} P[\mu](z) dV(z) \\ &=& dR^{2d} \int_{\partial \mathbb B^d} f(R\zeta) \overline{g(R\zeta)} P[\mu](R\zeta)  d\sigma(\zeta). 
\eeqn
\end{proposition}
\begin{proof}
For $h :\mathbb B^d \rar \mathbb C$ and $0 \leqslant R < 1,$ let $h_{\!_R}(z)=h(Rz),$ $z \in \overline{\mathbb B}^d,$ and note that  $\Delta h_{\!_R}(z)=R^2(\Delta h)(Rz).$ 
Since $P[\mu]$ is a harmonic function on $\mathbb B^d,$ by Lemma \ref{Rajeev} (with $C=R^2$),
\beqn
&& \!\! \!\! \sum_{j=1}^d \int_{R\mathbb B^d} \inp{\nabla z_j f}{\nabla z_jg}P[\mu](z)dV(z) - R^2 \int_{R\mathbb B^d}  \inp{\nabla f}{\nabla g} P[\mu](z)dV(z) \\ &=&\int_{R\mathbb B^d}  \Delta ((\|z\|^2-R^2)f\overline{g})P[\mu](z)dV(z) \\
&=& R^{2d} \int_{\mathbb B^d}   \Delta ((\|w\|^2-1)f_{\!_R}\overline{g}_{\!_R})P[\mu](Rw)dV(w) \\
&\overset{\eqref{Green}}=& \frac{dR^{2d}}{2} \int_{\partial \mathbb B^d}  \frac{\partial}{\partial r} \Big((r^2-1) f_{\!_R}(r\zeta)\overline{g}_{\!_R}(r\zeta)\Big)\Big|_{r=1} P[\mu](R\zeta) d\sigma(\zeta) \\
&=& dR^{2d} \int_{\partial \mathbb B^d}  f_{\!_R}(\zeta)\overline{g}_{\!_R}(\zeta) P[\mu](R\zeta)  d\sigma(\zeta).
\eeqn 
This completes the proof.
\end{proof}
\begin{remark} \label{R-F-fails} Let $\mu \in M_+(\partial \mathbb B^d).$ Then
the invariant Poisson integral $P_{\iota}[\mu]$ is an $\mathcal M$-harmonic function (see \cite[Section 4.3]{Ru-2}), which is not necessarily harmonic. The argument above shows that if Proposition \ref{Richter-char-R} holds for $P_{\iota}[\mu],$ then for $0 < R < 1$ and every holomorphic functions $f, g :\mathbb B^d \rar \mathbb C,$ 
\beqn
\int_{\mathbb B^d}   (\|w\|^2-1)f_{\!_R}(w)\overline{g}_{\!_R}(w) \Delta P_{\iota}[\mu](Rw)dV(w)=0.
\eeqn
Now, it may be concluded from Stone-Weierstrass and Riesz representation theorems that Proposition \ref{Richter-char-R} fails  for the invariant Poisson integral $P_{\iota}[\mu],$ unless $P_{\iota}[\mu]$ is harmonic. Similarly, one may see that Richter's formula fails for the invariant Poisson integral in general.
%Note that the formula \eqref{formula-f} extends to the functions in the ball algebra $A(\mathbb B^d),$ that is, functions holomorphic on the unit ball $\mathbb B^d,$ which extend continuously to $\overline{\mathbb B}^d.$ 
Since the invariant Poisson kernel is naturally linked to the theory of M$\ddot{\mbox{o}}$bius-harmonic (or $\mathcal M$-harmonic) functions arising from the notion of the invariant Laplacian (refer to \cite{Ru-2}), one may tempt to replace the gradient in the definition of invariant Dirichlet-type spaces by the invariant gradient. 
However, this does not yield a successful analogue of Dirichlet-type spaces even in dimension $d=1$ in the sense that the associated multiplication operator is not a $2$-isometry.
\end{remark} 

\begin{proof}[Proof of Theorem \ref{Richter-several}]
In view of Lemma \ref{Richter-char}, it suffices to derive \eqref{Richter's measure}. 
For $0 < R < 1,$ let $\mu_{\!_R}$ be the weighted surface area measure with weight $\mathsf w_{\!_R}(\zeta) = P[\mu](R\zeta),$ $\zeta \in \partial \mathbb B^d.$ 
%Since $\mathsf w_s$ extends to a harmonic function on an open neighborhood of $\overline{\mathbb B}^d$ (see \cite[Remark 3.3.10]{Ru-2}), by  the uniqueness part of Theorem \ref{DP}, we have
%\beq
%\label{P-weight} 
%P[\mu_s] =\mathsf w_s~\mbox{is a harmonic function in an open neighbourhood of}~\overline{\mathbb B}^d. 
%\eeq 
Note that the proof of \cite[Theorem 3.3.4(c)]{Ru-2} (asserted for invariant Poisson kernel $P_{\iota}(z, \zeta)$) essentially relies on variants of \eqref{P-nor-2}, \eqref{P-sym} and \eqref{pres-nbhd} for $P_{\iota}(z, \zeta).$ Hence
\cite[Theorem 3.3.4(c)]{Ru-2}  extends to the Euclidean Poisson kernel $P(z, \zeta)$ (with the same proof), and we may conclude that  
\beq
\label{weak-star-cgn-1}
\lim_{R \rar 1^{-}} \mu_{\!_R} = \mu~\mbox{in the weak*-topology of the dual of}~C(\partial \mathbb B^d).
\eeq 
Applying Proposition \ref{Richter-char-R} to $f(z)=z^{\alpha}$ and $g(z)=z^{\beta}$ for $\alpha, \beta \in \mathbb Z^d_+$ together with the dominated convergence theorem (twice) and \eqref{weak-star-cgn-1}  yields \eqref{Richter's measure}. 
\end{proof}

Here is a special case of Theorem \ref{Richter-several} (the case in which $\mu \in M_+(\partial \mathbb B^d)$ is the Dirac delta measure with point mass) of  independent interest.
\begin{corollary} 
For every $\zeta \in \partial \mathbb B^d,$  we have
\beqn
&& \sum_{|\gamma|=k} \frac{|\gamma|!}{\gamma!}  \int_{\mathbb B^d} \inp{\nabla z^{\gamma}p(z)}{\nabla z^{\gamma}q(z)} P(z, \zeta) dV(z)  \\ &=&  \int_{\mathbb B^d} \inp{\nabla p(z)}{\nabla q(z)} P(z, \zeta) dV(z) ~+~ kd \, p(\zeta) \overline{q(\zeta)}, \notag \\ 
 && \notag \hspace{5cm} p, q \in \mathbb C[z_1, \ldots, z_d], \quad k \in \mathbb Z_+.
\eeqn
\end{corollary} 

In the following section, we see the role of Richter's formula (including Proposition \ref{Richter-char-R}) in the study of Dirichlet-type spaces. In particular, we answer major half of Question \ref{Q1.1}. 
%\begin{corollary}
%Let $\mathsf w \in L^1(\partial \mathbb B^d).$ Then the weighted surface area measure with weight $\mathsf w$ belongs to $\mathscr{R\!\!M}_{\!\!+\!}(\partial \mathbb B^d).$
%\end{corollary}

\section{Joint $m$-isometries and Dirichlet-type spaces}

Let $\mathcal H$ be a complex Hilbert space. Let $\inp{\cdot}{\cdot}_{\mathcal H}$ and $\|\cdot\|_{\mathcal H}$ denote the inner-product  and norm on $\mathcal H,$ respectively. If no confusion occurs, we omit the subscript $\mathcal H$ from these notations.
Let $\mathcal B(\mathcal H)$ denote the $C^*$-algebra of bounded linear operators on $\mathcal H.$ By a {\it commuting $d$-tuple $T$ on $\mathcal H,$} we understand the $d$-tuple $(T_1, \ldots, T_d)$ consisting of commuting operators $T_1, \ldots, T_d$ in $\mathcal B(\mathcal H).$ 
Given a commuting $d$-tuple $T$ on ${\mathcal
H}$ and $\alpha=(\alpha_1, \ldots, \alpha_d) \in \mathbb Z^d_+,$  let $T^*$ denote the commuting $d$-tuple $(T^*_1, \ldots, T^*_d)$ and $T^{\alpha}$ denote the operator $T^{\alpha_1}_1 \cdots T^{\alpha_d}_d$ in $\mathcal B(\mathcal H).$ If $\mathcal M$ is a subspace of $\mathcal H$ such that $T_j \mathcal M \subseteq \mathcal M$ for every $j=1, \ldots, d,$ then we use the notation $T|_{\mathcal M}$  to denote the commuting $d$-tuple $(T_1|_{\mathcal M}, \ldots, T_d|_{\mathcal M})$ on $\mathcal M.$
With every commuting $d$-tuple $T$ on $\mathcal H,$ we associate the positive map $Q_T : \mathcal B(\mathcal H) \rar \mathcal B(\mathcal H)$ given by  \beqn  Q_T(X) \mathrel{\mathop:}=
\sum_{j=1}^d T^*_jXT_j, \quad X \in B(\mathcal H).\eeqn 
%We refer to $Q_T$ as the {\it spherical generating $1$-tuple associated with $T$}.
The operator $Q^n_T$
is inductively defined for all integers $n \geqslant 0$ through
the relations $Q^0_T(X)=X$ and $Q^n_T(X)=Q_T(Q^{n-1}_T(X)),$ $n \geqslant 1,$ $X \in B(\mathcal H).$ 
It is easy to see
that \beq \label{QT-n} Q^n_T(I)=\sum_{|\alpha|=n}\frac{|\alpha|!}{\alpha!} T^{*\alpha} T^{\alpha}, \quad n \in \mathbb Z_+. \eeq
%Let $Q_T$ be as given in \eqref{Q}. 
For $m \in \mathbb Z_+,$ let \beq \label{BMQT} B_m(T) :=
\sum_{n=0}^m (-1)^{n} \binom{m}{n}  Q^{n}_T(I). \eeq   If
$B_m(T)=0$, then $T$ is said to be a {\it joint $m$-isometry}. In $d=1,$ this notion first appeared in \cite{Ag}, and in the general case, it was formally introduced in \cite{GR}. The reader is referred to \cite{At, EP, GR, CY, HM, EL, Gu, Le} for examples and basic properties of joint $m$-isometries. We refer to joint $1$-isometry simply as {\it joint isometry or spherical isometry}.  Following \cite{Sh}, we say that $T$ is {\it joint $m$-concave} (resp. {\it joint $m$-convex}) if $(-1)^mB_m(T) \leqslant 0$ (resp. $(-1)^mB_m(T) \geqslant 0$).

Before we define Dirichlet-type spaces, we discuss some examples of joint $m$-isometries arising from a family of Hardy-Besov spaces. 
\begin{example} \label{exam-classical}
For a real number $p>0$, let 
\beqn
\kappa_p(z, w)=\frac{1}{(1- \inp{z}{w})^p}, \quad z, w \in
\mathbb B^d.
\eeqn
Note that $\kappa_p$ defines a positive definite kernel on $\mathbb B^d,$ which is holomorphic in $z$ and conjugate holomorphic in $w$.
Thus we may associate a reproducing kernel Hilbert space $\mathscr H_{p}$ of
holomorphic functions on the unit ball $\mathbb B^d$ with the reproducing kernel $\kappa_p$ (refer to \cite{PR} for this construction).  
Note that the monomials are orthogonal in $\mathscr H_p$ (see, for instance, \cite[Lemma 2.14]{CY}). Indeed, for any $\alpha, \beta  \in \mathbb Z^d_+,$ we have
\beq \label{Diri-formula}
\inp{z^{\alpha}}{z^{\beta}}_{\mathscr H_p} &=& \begin{cases}  \frac{\alpha! \Gamma(p)}{\Gamma(|\alpha|+p)} & \mbox{if~} \alpha = \beta, \\
0 & \mbox{otherwise},
\end{cases}
\eeq
where $\Gamma$ denotes the gamma function (see \cite[Equation (4.1)]{GR}).
In particular, for every $p > 0,$
$\mathscr H_{p}$ is contractively embedded into $\mathscr H_{p+1}.$
%If $M_{z, p}$ denotes the multiplication tuple on $\mathscr H_p$ then it is unitarily
%equivalent to the weighted shift $d$-tuple with weight sequence
%\beqn
%%\label{Sz-Be-Dru}
%w^{(i)}_{n, p}=\sqrt{\frac{n_i + 1}{|n|+p}}~(n \in \mathbb
%N^m, i=1, \cdots, m).
%\eeqn
The reproducing kernel Hilbert spaces $\mathscr H_{d}$ and
$\mathscr H_{1}$ are the {\it Hardy space} $H^2(\mathbb B^d)$ and
the {\it Drury-Arveson
space} $H^2_d,$ respectively. It may be deduced from \eqref{Diri-formula} that $z_j,$ $j=1, \ldots, d,$ is a multiplier for $\mathscr H_p.$ 
Let $\mathscr M_{z, p}$ denote the $d$-tuple of multiplication operators $\mathscr M_{z_1}, \ldots, \mathscr M_{z_d}$ in $\mathcal B(\mathscr H_p).$ It turns out that $\mathscr M_{z, p}$ is a joint $m$-isometry if and only if $p$ is a positive integer,
$p \leqslant d,$ and 
$m \geqslant d - p + 1$ 
(see \cite[Theorem 4.2]{GR}; see also the discussion following \cite[Theorem 5.3]{CY}).  
In particular, the multiplication tuple $\mathscr M_{z, d}$ ({\it
Szeg\"o $d$-shift}) turns out to be a joint isometry, while
$\mathscr M_{z, 1}$ ({\it Drury-Arveson $d$-shift}) is a  joint $d$-isometry.
%are commonly known as the {\it
%Szeg\"o $d$-shift}, the {\it Bergman $d$-shift}, the
%{\it Drury-Arveson $d$-shift} respectively. It is worth noting that 
\eop
\end{example}

%In the sequel, we need the following well-known fact.
%\begin{lemma} Let $p$ be any number bigger than $-1$ and $\beta \in \mathbb Z^d_+.$ Then
%\beq
%\label{I-p}
% \int_{\mathbb B^d} (1-\|z\|^2)^p|z^{\beta}|^2 dV(z) = 
%\Gamma(p+1) \frac{\beta! d!}{\Gamma(|\beta|+p+d+1)}.
%\eeq
%Moreoever, for $p > d,$ we have
%\beq \label{Diri-formula-new}
%\|z^{\beta}\|^2_{\mathscr H_p} = \frac{\Gamma(p)}{d!\Gamma(p-d)}\int_{\mathbb B^d} (1-\|z\|^2)^{p-d-1}
%|z^{\beta}|^2 dV(z).
%\eeq
%\end{lemma}
%\begin{proof}
%The formula \eqref{Diri-formula-new} may be deduced from \cite[Pg 43, Eqn (2.8)]{Z} and \eqref{Diri-formula}. This together with \eqref{Diri-formula} yields the formula \eqref{I-p}.
%\end{proof}

%Although we will be dealing with scalar-valued Dirichlet-type spaces in major half of this text, 
Following \cite{O} and considering the role of vector-valued Dirichlet-type spaces in the model theory for arbitrary $2$-isometries, we have chosen to work in the vector-valued set-up.

For a complex separable Hilbert space $\mathcal M,$ let $M_+(\partial \mathbb B^d, \mathcal B(\mathcal M))$ denote the cone of 
 $\mathcal B(\mathcal M)$-valued semispectral measures $F$ on the unit sphere $\partial \mathbb B^d$ in $\mathbb C^d.$ In case $\mathcal M=\mathbb C,$ following our earlier notation, we denote $M_+(\partial \mathbb B^d, \mathcal B(\mathcal M))$ simply by $M_+(\partial \mathbb B^d).$
Here by a $\mathcal B(\mathcal M)$-valued {\it semispectral measure} $F$ on $\partial \mathbb B^d,$ we understand a set function on $\partial \mathbb B^d$ which is finitely additive with values being positive operators in $\mathcal B(\mathcal M)$ such that $\inp{F(\cdot)x}{y}$  defines a complex regular Borel measure on $\partial \mathbb B^d$ for every $x, y \in \mathcal M.$ For $F \in M_+(\partial \mathbb B^d, \mathcal B(\mathcal M)),$
consider now the {\it Poisson integral} $P[F]$ of $F$ given by
\beqn
P[F](z)=\int_{\partial \mathbb B^d} P(z, \zeta) dF(\zeta), \quad z \in \mathbb B^d.
\eeqn
Let $H^2_{\mathcal M}(\mathbb B^d)$ denote the Hardy space  of $\mathcal M$-valued holomorphic functions $f(z)=\sum_{\alpha \in \mathbb Z^d_+} \mathsf a_{\alpha}z^{\alpha}$ on $\mathbb B^d$ endowed with the norm governed by
\beqn
\|f\|^2_{H^2_{\mathcal M}(\mathbb B^d)} = \sum_{\alpha \in \mathbb Z^d_+} \|\mathsf a_{\alpha}\|^2_{\mathcal M} \|z^{\alpha}\|^2_{H^2(\mathbb B^d)},
\eeqn 
where $\{\mathsf a_{\alpha}\}_{\alpha \in \mathbb Z^d_+} \subseteq \mathcal M.$
The {\it Dirichlet-type space $\mathscr D(F)$ associated with $F$} is defined as the complex vector space of $\mathcal M$-valued holomorphic functions $f : \mathbb B^d \rar \mathcal M$ such that $\|f\|_{\mathscr D(F)}$ is finite, where
\beqn
\|f\|^2_{\mathscr D(F)} := \|f\|^2_{H^2_{\mathcal M}(\mathbb B^d)} +  \frac{1}{d}\int_{\mathbb B^d} \sum_{j=1}^d \Big \langle{P[F](z) \frac{\partial f}{\partial z_j}},{\frac{\partial f}{\partial z_j}} \Big \rangle_{\mathcal M} dV(z),
\eeqn
where, by abuse of notation, $\frac{\partial f}{\partial z_j}$ evaluated at the point $z$ is denoted by $\frac{\partial f}{\partial z_j}.$
%{\blue For a holomorphic function $f : \mathbb B^d \rar \mathcal M,$ set $\|\nabla f\|^2 := \sum_{j=1}^d \Big \langle{ \frac{\partial f}{\partial z_j}},{\frac{\partial f}{\partial z_j}} \Big \rangle_{\mathcal M}.$}

%Note that 
%in case $\phi_F[z]$ is the identity operator $I_{\mathcal M},$ then \beqn \|f\|^2_{\mathscr D(F)} := \|f(0)\|^2_{\mathcal M} +  \int_{\mathbb B^d} \|\nabla f\|^2dV(z).\eeqn

\begin{remark} \label{Hardy-inclusion} 
We note the following:
\begin{enumerate}
\item[(i)] The complex vector space $\mathscr D(F)$  contains all $\mathcal M$-valued polynomials in $z_1, \ldots, z_d,$ and it is a subspace of the Hardy space $H^2_{\mathcal M}(\mathbb B^d).$  If for $j=1, \ldots, d,$ $\mathscr M_{z_j}$ denotes the linear operator of multiplication by the coordinate function $z_j,$ then
the $d$-tuple $\mathscr M_z=(\mathscr M_{z_1}, \ldots, \mathscr M_{z_d})$ is defined and commuting on the space of $\mathcal M$-valued polynomials. 
\item[(ii)]  As in the one variable scalar case (refer to \cite[Chapter 1]{EKMR}), it may be deduced from the Fatou's Lemma and the reproducing property of $H^2_{\mathcal M}(\mathbb B^d)$ that $\mathscr D(F)$ is a 
reproducing kernel Hilbert space. Further, every $z \in \mathbb B^d$ is a bounded point evaluation for $\mathscr D(F)$, and if $\kappa : \mathbb B^d \times \mathbb B^d \rar \mathcal B(\mathcal M)$ denotes the reproducing kernel for $\mathscr D(F),$ then $\kappa(\cdot, 0)=I_{\mathcal M},$ where $I_{\mathcal M}$ denotes the identity operator on $\mathcal M.$  The latter one follows from $\inp{f}{x}_{\mathscr D(F)}=\inp{f}{x}_{H^2_{\mathcal M}(\mathbb B^d)}= \inp{f(0)}{x}_{\mathcal M},$ $f \in \mathscr D(F),$ $x \in \mathcal M.$ 
% and the joint kernel of $\mathscr M^*_z$ contains the space $\mathcal M$ of constant functions in 
% $\mathscr D(F),$
\item[(iii)] Unlike the one-dimensional case, there are various possibilities for the definition of a Dirichlet-type space in general. Indeed,  in the definition of $\mathscr D(F),$ one  may replace $H^2_{\mathcal M}(\mathbb B^d)$ by any Hilbert space $\mathscr H$ of $\mathcal M$-valued holomorphic functions for which the $d$-tuple $\mathscr M_z$ of multiplication operators on $\mathscr H$ is a joint isometry (cf. \cite[Equation (1.4)]{Ry}). However, we do not take up this notion here.
\end{enumerate}
\end{remark}

%Unlike the case of $d=1,$ there is no unique space of scalar-valued holomorphic  functions for which the multiplication tuples $\mathscr M_z$ is a spherical isometry. In particular, in the definition of the Dirichlet-type spaces, one may replace the Hardy space of the unit ball by

%\begin{definition} Let $\mathcal M$ be a complex separable Hilbert space and $F$ is a $\mathcal B(\mathcal M)$-valued semi-spectral measure on the unit sphere $\partial \mathbb B^d$ in $\mathbb C^d$ {\red such that $F(\partial \mathbb B^d)$ is an invertible operator in $\mathcal B(\mathcal M).$} For any holomorphic function $f : \mathbb B^d \rar \mathcal M,$ set
%\beqn
%\|f\|^2_{\mathscr D(F)} = \|f\|^2_{H^2_{\mathcal M}(\partial \mathbb B^d)} + \|f\|^2_{\circ} - \|f(0)\|^2_{\mathcal M}.
%\eeqn
%The {\it $\mathcal M$-valued Dirichlet-type space} $\mathscr D(F)$ is the complex vector space of $\mathcal M$-valued holomorphic functions $f : \mathbb B^d \rar \mathcal M$ such that $\|f\|_{\mathscr D(F)}$ is finite.
%\end{definition}
%\begin{remark} 
% Note that by \eqref{Hardy-circle}, 
%$\|\cdot\|_{\circ}$ is equivalent to the norm $\|\cdot\|_{\mathscr D(F)}.$ Hence the conclusions in (i) and (ii) of Proposition \ref{Hardy-inclusion} also hold for the Dirichlet-type space $\mathscr D(F).$
%\end{remark}
%(this is always satisfied in case $\mathcal M$ is a one-dimensional space and $F$ is a nonzero finite positive Borel measure on the unit sphere).

The following vector-valued analogue of Proposition \ref{Richter-char-R} is immediate from its scalar counterpart.

\begin{proposition} \label{Richter-char-R-vector}
Let $F \in M_+(\partial \mathbb B^d, \mathcal M)$ and $0 \leqslant R < 1.$ Then, for every $\mathcal M$-valued holomorphic functions $f, g :\mathbb B^d \rar \mathcal  M,$ we have
\beqn
 \displaystyle \sum_{i, j=1}^d   \int_{R\mathbb B^d} \!\!\! \Big \langle{P[F](z)\frac{\partial z_if}{\partial z_j}},{\frac{\partial z_ig}{\partial z_j}}\Big \rangle  dV(z)  -  R^2 \sum_{j=1}^d
 \int_{R\mathbb B^d} \!\!\! \Big \langle{P[F](z)\frac{\partial f}{\partial z_j}},{\frac{\partial g}{\partial z_j}} \Big \rangle  dV(z) \\ ~=~ dR^{2d} \int_{\partial \mathbb B^d} \inp{P[F](R\zeta)f(R\zeta)}{g(R\zeta)}  d\sigma(\zeta). \quad \quad \quad \quad
\eeqn
\end{proposition}
\begin{proof}
%Since $f, g$ converge uniformly on $R\mathbb B^d,$ $0 < R < 1,$ by the linearity of the integral, it suffices to verify the above formula for  
%Since every $\mathcal M$-valued holomorphic function can be written as 
This follows from the scalar case (see Proposition \ref{Richter-char-R}) applied to the complex measures $\inp{F(\cdot)x}{y},$ $x, y \in \mathcal M.$
\end{proof}
%\begin{theorem} \label{Richter-several-vector}
%%Let $F$ be a $\mathcal B(\mathcal M)$-valued semispectral measure on the unit sphere $\partial \mathbb B^d$ in $\mathbb C^d.$ 
%Let $F \in M_+(\partial \mathbb B^d, \mathcal B(\mathcal M)).$
%Then, for every $\mathcal M$-valued polynomials $p, q$ and a non-negative integer $k,$
%\beqn 
%&& \notag  \sum_{|\gamma|=k} \frac{|\gamma|!}{\gamma!}  \int_{\mathbb B^d} \sum_{j=1}^d \Big \langle{P[F](z) \frac{\partial z^{\gamma}p}{\partial z_j}},{\frac{\partial z^{\gamma}q}{\partial z_j}} \Big \rangle_{\mathcal M}   dV(z)  \\ &=&  \int_{\mathbb B^d} \sum_{j=1}^d \Big \langle{P[F](z) \frac{\partial p}{\partial z_j}},{\frac{\partial q}{\partial z_j}} \Big \rangle_{\mathcal M} dV(z) ~+~ kd \int_{\partial \mathbb B^d} p(\zeta) \overline{q(\zeta)} dF(\zeta).
%\eeqn
%\end{theorem} 

We can now answer Question \ref{Q1.1}(b).

\begin{theorem} \label{R-F-role-1}
Let $F \in M_+(\partial \mathbb B^d, \mathcal B(\mathcal M)).$
Assume that for $k=1, \ldots, d,$ $z_k$ is a multiplier for $\mathscr D(F).$ Then the $d$-tuple $\mathscr M_z$ on $\mathscr D(F)$ is a joint $2$-isometry.
\end{theorem}
\begin{proof} Let $f \in \mathscr D(F).$ Thus, $z_j f \in \mathscr D(F),$ $j=1, \ldots, d,$ and hence
by Proposition \ref{Richter-char-R-vector} and the dominated convergence theorem, 
\beq \label{limit-R-Richter}
\sum_{j=1}^d \|z_jf\|^2_{\mathscr D(F)} - \|f\|^2_{\mathscr D(F)} = \lim_{R \rar 1^{-}} \int_{\partial \mathbb B^d} \inp{P[F](R\zeta)f(R\zeta)}{g(R\zeta)}  d\sigma(\zeta). 
\eeq
It follows that 
\beqn
&& \|f\|^2_{\mathscr D(F)} - 2\sum_{j=1}^d \|z_jf\|^2_{\mathscr D(F)} + \sum_{j, k=1}^d \|z_jz_kf\|^2_{\mathscr D(F)}  \\ &=&  \bigg(\|f\|^2_{\mathscr D(F)} -  \sum_{j=1}^d \|z_jf\|^2_{\mathscr D(F)}\bigg) - \sum_{j=1}^d \bigg(\|z_jf\|^2_{\mathscr D(F)} - \sum_{k=1}^d \|z_k(z_jf)\|^2_{\mathscr D(F)}\bigg) \\
&=& \lim_{R \rar 1^{-}} \int_{\partial \mathbb B^d} \inp{P[F](R\zeta)f(R\zeta)}{f(R\zeta)}  d\sigma(\zeta)  \\ &-& \sum_{j=1}^d \lim_{R \rar 1^{-}} \int_{\partial \mathbb B^d} R^2 \inp{P[F](R\zeta)\zeta_j f(R\zeta)}{\zeta_j f(R\zeta)}  d\sigma(\zeta),
\eeqn
which is easily seen to be equal to zero. 
\end{proof}

The question of when the coordinate functions are multipliers for the Dirichlet-type spaces can be related to the notion of Carleson measure (the reader is referred to \cite{VW} for the definition of Carleson measure).
\begin{proposition} \label{prop-multi}
%Let $F$ be a $\mathcal B(\mathcal M)$-valued semispectral measure on the unit sphere $\partial \mathbb B^d$ in $\mathbb C^d.$ 
Let $F \in M_+(\partial \mathbb B^d, \mathcal B(\mathcal M)).$
Then the following statements are equivalent$:$
\begin{enumerate}
\item[$(i)$] For $k=1, \ldots, d,$ $z_k$ is a multiplier for $\mathscr D(F).$
\item[$(ii)$] 
There exists a positive constant $C$ such that 
\beqn
\int_{\mathbb B^d} \Big \langle{P[F](z) f(z)},{f(z)} \Big \rangle_{\mathcal M} dV(z) \leqslant C \|f\|^2_{\mathscr D(F)}, \quad f \in \mathscr D(F).
\eeqn
\item[$(iii)$] 
There exists a positive constant $\tilde{C}$ such that 
\beqn
%\limsup_{R \rar 1^{-}} 
\sup_{0 < R < 1} \int_{\partial \mathbb B^d} \inp{P[F](R\zeta)f(R\zeta)}{f(R\zeta)}  d\sigma(\zeta) \leqslant \tilde{C} \|f\|^2_{\mathscr D(F)}, \quad f \in \mathscr D(F).
\eeqn
\end{enumerate} 
%\uwam{add (iv) which is vacuously true in $d=1$}
%If $(i)$ holds, then the $d$-tuple $\mathscr M_z$ is a joint $2$-isometry.
%\uwam{fails for $d \geq 2$?}
\end{proposition}
\begin{proof}
Let $\mbox{Hol}(\mathbb B^d, \mathcal M)$ denote the vector space of $\mathcal M$-valued holomorphic functions on $\mathbb B^d.$ For $f \in \mbox{Hol}(\mathbb B^d, \mathcal M),$ define $$\|f\|^2_{\star} = \int_{\mathbb B^d} \big \langle{P[F](z) f(z)},{f(z)} \big \rangle_{\mathcal M} dV(z),$$ and note that $\|\cdot\|_{\star}$ defines (possibly an extended real-valued) semi-norm on $\mbox{Hol}(\mathbb B^d, \mathcal M).$  In particular, for $g, h \in \mbox{Hol}(\mathbb B^d, \mathcal M),$
\beq \label{seminorm}
\mbox{if $\|g\|_{\star} < \infty,$ then $\|g+h\|_{\star} < \infty$ if and  only  if $\|h\|_{\star} < \infty.$}
\eeq 
Note that $f \in \mathscr D(F)$ if and only if $f \in H^2_{\mathcal M}(\mathbb B^d)$ and $\|\frac{\partial f}{\partial z_j} \|_{\star} < \infty$ for every $j=1, \ldots, d.$ This yields in particular that for any $f \in \mathscr D(F),$ \beq \label{multi-L2} \Big\|z_k\frac{\partial f}{\partial z_j} \Big\|_{\star} < \infty, \quad j, k=1, \ldots, d. \eeq  Further,
for any $k=1, \ldots, d,$ we have
\beqn 
&& \int_{\mathbb B^d} \sum_{j=1}^d \Big \langle{P[F](z) \frac{\partial (z_kf)}{\partial z_j}},{\frac{\partial (z_kf)}{\partial z_j}} \Big \rangle_{\mathcal M} dV(z) \\ &=& \sum_{\underset{j \neq k}{j=1}}^d \Big\|z_k\frac{\partial f}{\partial z_j}\Big\|^2_{\star} + \Big\|z_k\frac{\partial f}{\partial z_k} + f \Big\|^2_{\star}. 
\eeqn
It now follows from \eqref{seminorm} and \eqref{multi-L2} that for any $f \in \mathscr D(F),$
\beq \label{zk-F-norm}
\mbox{$z_kf \in \mathscr D(F)$ for every $k=1, \ldots, d$ if and only if $\|f\|_{\star} < \infty.$}
\eeq
This yields (ii) $\Rightarrow$ (i).

To see (i) $\Rightarrow$ (ii), note first that by Remark \ref{Hardy-inclusion}(i), 
$\mathscr D(F) \subseteq H^2_{\mathcal M}(\mathbb B^d),$ and hence
we may consider the normed linear space $\mathscr H=\mathscr D(F)$ endowed with the norm $\max\{\|\cdot\|_{\star}, \|\cdot\|_{H^2_{\mathcal M}(\mathbb B^d)}\}.$ In view of \eqref{zk-F-norm}, the map $f \mapsto f$ from $\mathscr D(F)$ into the completion of $\mathscr H$ is well-defined, and hence 
one may apply the closed graph theorem to obtain (ii). 

The implication (i) $\Rightarrow$ (iii) is immediate from \eqref{limit-R-Richter}. The other implication follows from Proposition \ref{Richter-char-R-vector} and the monotone convergence theorem.
\end{proof}
\begin{remark}
In case of $d=1,$ (i) above always holds (see \cite[Theorem 3.6]{Ri} for $\dim \mathcal M =1$ and \cite[Theorem 3.1]{O} for the general case). In case $\mathcal M=\mathbb C,$ the condition (ii) above says that the Dirichlet space $\mathscr D(F)$ is boundedly embedded into the weighted Bergman space with weight being $P[F].$  In turn,  this is equivalent to the assertion that the weighted volume measure with weight being $P[F]$ is a Carleson measure for $\mathscr D(F)$ (see 
\cite[Theorem 2]{VW} for a characterization of Carleson measures for Besov-Sobolev spaces on $\mathbb B^d$).
%\cite[Section 3]{ARSW} for the role of Carleson measures in the theory of the Dirichlet space).
% {\red Finally, note that 
%\beq \label{ii-mono}
%\|z^{\alpha}x\|_{F} \leqslant \max \left\{\|F(\partial \mathbb B^d)\|, 1 \right\} \|z^{\alpha}x\|_{\mathscr D(F)}, \quad \alpha \in \mathbb Z^d_+, x \in \mathcal M.
%\eeq 
%To see this, first note that by \eqref{P-nor}, for any $x \in \mathcal M,$
%\beqn
%\|x\|^2_F=\int_{\mathbb B^d} \Big \langle{P[F](z) x},{x} \Big \rangle_{\mathcal M} dV(z) = \inp{F(\partial \mathbb B^d)x}{x} \leqslant  \|F(\partial \mathbb B^d)\| \|x\|^2_{\mathcal M}.
%\eeqn
%Moreover, 
%\beqn 
% |z^{\alpha}|^2 \leqslant   \sum_{j=1}^d \alpha^2_j  |z^{\alpha-\varepsilon_j}|^2, \quad z \in \mathbb B^d, ~\alpha \in \mathbb Z^d_+ \setminus \{0\},
%\eeqn
%it follows from the non-negativity of $P[F](z),$ $z \in \mathbb B^d$ that \eqref{ii-mono} holds for all non-zero monomials with the constant $C$ being equal to $1.$}
\end{remark}

The following provides a partial answer to Question \ref{Q1.1}(a).
\begin{corollary} \label{Coro-multi}
For a bounded measurable function $\mathsf w : \partial \mathbb B^d \rar [0, \infty),$ consider the weighted surface area measure $\mu_{\mathsf w}$ with weight function $\mathsf w.$  Then the $d$-tuple $\mathscr M_z$ on $\mathscr D(\mu_{\mathsf w})$ is a joint $2$-isometry.
\end{corollary}
\begin{proof}
In view of Proposition \ref{prop-multi} and Theorem \ref{R-F-role-1}, it suffices to check that there exists a positive constant $C_{\mathsf w}$ such that
\beqn
\int_{\mathbb B^d} |f(z)|^2 P[\mu_{\mathsf w}](z) dV(z) \leqslant C_{\mathsf w} \|f\|^2_{\mathscr D(\mu_{\mathsf w})}, \quad f \in \mathscr D(\mu_{\mathsf w}).
\eeqn
Since 
$\mathscr D(\mu_{\mathsf w}) \subseteq H^2(\mathbb B^d)= \mathscr H_d$ and  $\mathscr H_d$ is contractively embedded into $\mathscr H_{d+1}$ (see Example \ref{exam-classical} and Remark \ref{Hardy-inclusion}(i)) and $P[\mu_{\mathsf w}]$ is bounded above by $\sup \mathsf w,$ we obtain the desired inequality with $C_{\mathsf w}=\sup \mathsf w.$
\end{proof}

We do not know whether the weighted volume measure with weight being the Poisson integral of an arbitrary finite positive Borel measure $\mu$ on $\partial \mathbb B^d$, $d \geqslant 2$ is a Carleson measure for $\mathscr D(\mu).$ 

\section{Rotation invariant and pluriharmonic measures}
%\section{Rotation invariant measures}

In this section, we study Dirichlet-type spaces arising from two families of positive finite Borel measures on the unit sphere. Before we discuss the case of rotation invariant measures, we recall the notion of $\mathbb T^d$-invariance. 

A subset $\Omega$ of $\mathbb C^d$ is {\it $\mathbb T^d$-invariant} if for every $\theta=(\theta_1, \ldots, \theta_d) \in \mathbb R^d,$ \beqn e^{i \theta} \cdot \Omega:=\{(e^{i \theta_1}z_1, \ldots, e^{i \theta_d}z_d) : (z_1, \ldots, z_d) \in \Omega\} \subseteq \Omega. \eeqn
A positive finite Borel measure $\mu$ on a $\mathbb T^d$-invariant subset $\Omega$ of $\mathbb C^d$ is said to be {\it $\mathbb T^d$-invariant} if  for every Borel subset $A$ of $\Omega,$
\beqn
\mu(e^{i \theta} \cdot A) = \mu(A), \quad \theta=(\theta_1, \ldots, \theta_d) \in \mathbb R^d.
\eeqn
%Similarly, one may define the {\it $\mathbb T^d$-invariant} function on a  $\mathbb T^d$-invariant set. 

\begin{lemma} \label{lem-ortho} Let $\mu \in M_+(\partial \mathbb B^d).$
Then $\mu$ is $\mathbb T^d$-invariant if and only if the monomials are orthogonal in $\mathscr D(\mu).$ 
%In this case, the Poisson integral $P[\mu]$ is $\mathbb T^d$-invariant.
\end{lemma}
\begin{proof} 
By Theorem \ref{Richter-several}, the orthogonality of monomials in $\mathscr D(\mu)$ implies the orthogonality of monomials in the space $L^2(\mu)$ of square-integrable functions on $\partial \mathbb B^d,$    and hence by \cite[Lemma 2.3]{CK}, $\mu$ is $\mathbb T^d$-invariant.  
Conversely, if $\mu$ is $\mathbb T^d$-invariant, then
a routine verification shows that $P[\mu]dV$ is $\mathbb T^d$-invariant, and hence the monomials are orthogonal in $\mathscr D(\mu).$
\end{proof}

To show that Dirichlet-type spaces associated with rotation-invariant measures support joint $2$-isometries, we need another lemma.
\begin{lemma}
Let $\mu \in M_+(\partial \mathbb B^d)$ and let $\mathscr D(\mu)$ be the associated Dirichlet-type space. For $f \in \mathscr D(\mu),$ set 
%$\|f\|^2_{\circ} := \|f\|^2_{\mathscr D(\mu)} - \|f\|^2_{H^2(\partial \mathbb B^d)} + |f(0)|^2.$
\beqn
\|f\|^2_{\circ} := |f(0)|^2 + \frac{1}{d} \int_{\mathbb B^d} \|\nabla f (z)\|^2 P[\mu](z)   dV(z).  
\eeqn
Then we have
\beq
\label{mono-estimate}
\sum_{k=1}^d \|z^{\alpha+ \varepsilon_k}\|^2_{\circ} \leqslant \max \Big\{2(1+d), 
\frac{\mu(\partial \mathbb B^d)}{d} \Big\} \, \|z^{\alpha}\|^2_{\circ}, \quad \alpha \in \mathbb Z^d_+.
\eeq
\end{lemma}
\begin{proof}
By the mean value property for harmonic functions (see, for instance, \cite[(3.1.13)]{Si}), we have
%\cite[Pg 40]{Kr},
%\beq 
%\label{multi}
% \sum_{k=1}^d \|z_k\|^2_{\circ} &=&  \frac{1}{d} \int_{\mathbb B^d} P[\mu](z)  dV(z) \notag \\ &=& \frac{1}{d} \frac{\pi^{d/2}}{\Gamma(d/2+1)} P[\mu][0] =  \frac{1}{d}\frac{\pi^{d/2}\mu(\partial \mathbb B^d)}{\Gamma(d/2+1)}. 
%\eeq
\beq 
\label{multi}
 \sum_{k=1}^d \|z_k\|^2_{\circ} =  \frac{1}{d} \int_{\mathbb B^d} P[\mu](z)  dV(z) = \frac{P[\mu][0]}{d}   =  \frac{\mu(\partial \mathbb B^d)}{d}.
\eeq
Let $\alpha \in \mathbb Z^d_+ \setminus \{0\}$ and 
note that
\beq \label{esti-mono}
 |z^{\alpha}|^2 \leqslant   \sum_{j=1}^d \alpha^2_j  |z^{\alpha-\varepsilon_j}|^2, \quad z \in \mathbb B^d.
\eeq
Let $\delta_{jk}$ denote the Kronecker delta function, and note that
\beqn
 \sum_{k=1}^d \|z^{\alpha+ \varepsilon_k}\|^2_{\circ} &=&  \frac{1}{d}\sum_{k=1}^d  \int_{\mathbb B^d}  \sum_{j=1}^d (\alpha_j +\delta_{jk})^2 |z^{\alpha -\varepsilon_j + \varepsilon_k}|^2   P[\mu](z)  dV(z)  \\ & \leqslant & \frac{2}{d} \sum_{k=1}^d  \int_{\mathbb B^d}   \sum_{j=1}^d (\alpha^2_j +\delta_{jk}) |z^{\alpha -\varepsilon_j + \varepsilon_k}|^2 P[\mu](z) dV(z) \\ &= &
 \frac{2}{d}   \int_{\mathbb B^d}   \|z\|^2 \sum_{j=1}^d \alpha^2_j |z^{\alpha -\varepsilon_j}|^2  P[\mu](z) dV(z)  \\ &+&  2   \int_{\mathbb B^d}     |z^{\alpha}|^2   P[\mu](z) dV(z) \\
& \leqslant  & 2(1+ d)\|z^{\alpha}\|^2_{\circ} , 
\eeqn
where we used the estimate \eqref{esti-mono} in the last step.
Combining this with \eqref{multi}, we obtain \eqref{mono-estimate}.
\end{proof}

The following lists several properties of Dirichlet-type spaces associated with $\mathbb T^d$-invariant measures.
\begin{proposition} \label{prop-rotation-invariant}
Let $\mu \in M_+(\partial \mathbb B^d)$ and let $\mathscr D(\mu)$ be the associated Dirichlet-type space. 
%For $f \in \mathscr D(\mu),$ set 
%%$\|f\|^2_{\circ} := \|f\|^2_{\mathscr D(\mu)} - \|f\|^2_{H^2(\partial \mathbb B^d)} + |f(0)|^2.$
%\beqn
%\|f\|^2_{\circ} := |f(0)|^2 + \frac{1}{d} \int_{\mathbb B^d} \|\nabla f (z)\|^2 P[\mu](z)   dV(z).  
%\eeqn
%Then we have
%\beq
%\label{mono-estimate}
%\sum_{k=1}^d \|z^{\alpha+ \varepsilon_k}\|^2_{\circ} \leqslant \max \Big\{2(1+d), 
%\frac{\mu(\partial \mathbb B^d)}{d} \Big\} \, \|z^{\alpha}\|^2_{\circ}, \quad \alpha \in \mathbb Z^d_+.
%\eeq
%In addition,  
If
 $\mu$ is $\mathbb T^d$-invariant,
then the following statements are true$:$
\begin{enumerate}
\item[$(i)$] the monomials $z^\alpha,$ $\alpha \in \mathbb Z^d_+$ form an orthogonal basis for $\mathscr D(\mu),$ 
\item[$(ii)$] for every $j=1, \ldots, d,$ $z_j$ is a multiplier for $\mathscr D(\mu),$ 
\item[$(iii)$] the $d$-tuple $\mathscr M_z$ on $\mathscr D(\mu)$ is a joint $2$-isometry,
\item[$(iv)$] if $\inp{f \circ U}{g \circ U}_{\mathscr D(\mu)}=\inp{f}{g}_{\mathscr D(\mu)}$ for every unitary $d \times d$ matrix $U$ $($considered as a function from $\mathbb B^d$ onto $\mathbb B^d)$, then $\mu,$ up to a scalar multiple, is the surface area measure $\sigma.$  
\end{enumerate}
\end{proposition}
\begin{proof}
Assume now that $\mu$ is $\mathbb T^d$-invariant.
%Note that $P[\mu]$ is invariant under the action of $\mathbb T^d$, and hence
%By Lemma \ref{lem-ortho}, the monomials are orthogonal in $\mathscr D(\mu).$
Let $f(z)=\sum_{\alpha \in \mathbb Z^d_+} \hat{f}(\alpha) z^{\alpha} \in \mathscr D(\mu)$ for some $\hat{f}(\alpha) \in \mathbb C$ and let $0 < R < 1.$  Since $f$ is uniformly convergent on $R\mathbb B^d,$
by Lemma \ref{lem-ortho},
\beqn
&& \int_{R\mathbb B^d} \|\nabla f (z)\|^2 P[\mu](z)   dV(z) \\ &=& 
\int_{R\mathbb B^d}  \sum_{j=1}^d \alpha_j \beta_j \sum_{\alpha, \beta \in \mathbb Z^d_+}\hat{f}(\alpha) \overline{\hat{f}(\beta)} z^{\alpha-\varepsilon_j} \overline{z}^{\beta-\varepsilon_j} P[\mu](z) dV(z) \\
&=& \sum_{\alpha \in \mathbb Z^d_+}  |\hat{f}(\alpha)|^2  \int_{R\mathbb B^d}   \sum_{j=1}^d \alpha^2_j |z^{\alpha-\varepsilon_j}|^2  P[\mu](z)   dV(z) 
\\
&=& \sum_{\alpha \in \mathbb Z^d_+}  |\hat{f}(\alpha)|^2  \int_{R\mathbb B^d}   \|\nabla z^{\alpha}\|^2  P[\mu](z)   dV(z).
\eeqn
Letting $R \rar 1^{-}$ on both sides, we obtain \beq \label{norm-f} \|f\|^2_{\mathscr D(\mu)} = \sum_{\alpha \in \mathbb Z^d_+}  |\hat{f}(\alpha)|^2 \|z^{\alpha}\|^2_{\mathscr D(\mu)}. \eeq
In particular, the sequence of partial sum  $\sum_{|\alpha|\leqslant n} \hat{f}(\alpha) z^{\alpha}$ of $f$ converges to $f$ in $\mathscr D(\mu).$
Further, \eqref{norm-f} combined with \eqref{mono-estimate} yields that each $z_j,$ $j=1, \ldots, d$ is a multiplier for $\mathscr D(\mu),$ and hence the multiplication operator $\mathscr M_{z_j}$ is bounded on $\mathscr D(\mu).$ Thus, by Theorem \ref{R-F-role-1}, $\mathscr M_z$ on $\mathscr D(\mu)$ is a joint $2$-isometry. Finally, if $\inp{f \circ U}{g \circ U}_{\mathscr D(\mu)}=\inp{f}{g}_{\mathscr D(\mu)}$ for every unitary $d \times d$ matrix $U,$ then by Theorem \ref{Richter-several}, $$\int_{\partial \mathbb B^d} p \circ U(\zeta) \overline{q \circ U(\zeta)}d\mu(\zeta) = \int_{\partial \mathbb B^d} p(\zeta) \overline{q(\zeta)}d\mu(\zeta), \quad p, q \in \mathbb C[z_1, \ldots, z_d].$$
By Stone-Weierstrass and Riesz representation theorems, $\mu$ is invariant under the action of the unitary group. 
Hence, by \cite[Remark, Pg 16]{Ru-2}, 
$\mu$ is a scalar multiple of $\sigma.$
\end{proof}

We now exhibit a family of rotation invariant measures for which the Dirichlet-type spaces and the associated multiplication tuples can be described explicitly (it is worth noting that, up to a scalar multiple, the only rotation invariant measure on the unit circle is the arc-length measure).
\begin{example}
Let $\lambda \in \mathbb R$ and $c=(c_1, \ldots, c_d) \in \mathbb R^d$ be such that $\lambda > \max_{j=1}^d |c_j|$ and $\sum_{j=1}^d c_j=0.$ 
Consider the weighted surface area measure $\mu_{\lambda, c}$ on $\partial \mathbb B^d$
with weight function $w_{\lambda, c}$ given by
\beqn
w_{\lambda, c}(z) := \lambda + \sum_{j=1}^d c_j |z_j|^2, \quad z \in  \mathbb C^d.
\eeqn 
Clearly, $w_{\lambda, c}|_{\partial \mathbb B^d}$ is positive and $\mu_{\lambda, c}$ is $\mathbb T^d$-invariant.  Consider the Dirichlet-type space $\mathscr D(\mu_{\lambda, c})$ associated with $\mu_{\lambda, c}.$
It follows from Proposition \ref{prop-rotation-invariant} that $\big\{e_{\alpha} := z^\alpha /\|z^{\alpha}\|_{\mathscr D(\mu_{\lambda, c})} : \alpha \in \mathbb Z^d_+ \big\}$ is an orthonormal basis for $\mathscr D(\mu_{\lambda, c})$ and the $d$-tuple $\mathscr M_z$ on $\mathscr D(\mu_{\lambda, c})$ defines a joint $2$-isometry.  Moreover, $\mathscr M_z$ is a weighted multishift for some weights $ \mathbf w=\big\{\mathbf w^{(j)}_{\alpha} : \alpha \in \mathbb Z^d_+, ~j=1, \ldots, d \big\} \subseteq (0, \infty)$, that is, 
\beqn
\mathscr M_{z_j}e_{\alpha} = \mathbf w^{(j)}_{\alpha} e_{\alpha + \varepsilon_j}, \quad \alpha \in \mathbb Z^d_+, ~j=1, \ldots, d.
\eeqn  
To compute the weights $\mathbf w$ of $\mathscr M_z,$ note first that 
$ w_{\lambda, c}$ defines a harmonic function on $\mathbb C^d,$ and hence by the uniqueness part of Theorem \ref{DP},
$P[\mu_{\lambda, c}]=w_{\lambda, c}$ on $\mathbb B^d.$ It now follows from \eqref{Diri-formula} and the assumption $\sum_{j=1}^dc_k=0$ that
\beqn
\|z^{\alpha}\|^2_{\mathscr D(\mu_{\lambda, c})} &=& \|z^{\alpha}\|^2_{H^2(\mathbb B^d)} + \frac{1}{d}\int_{\mathbb B^d}    \|\nabla z^{\alpha}\|^2  w_{\lambda, c}(z) dV(z) \\ &=& \frac{\alpha!(d-1)!}{(|\alpha|+d-1)!} + \lambda \sum_{j=1}^d \alpha^2_j \frac{(\alpha - \varepsilon_j)!(d-1)!}{(|\alpha|+d-1)!} \\ 
&+& \sum_{k=1}^d c_k \sum_{j=1}^d \alpha^2_j \frac{(\alpha -\varepsilon_j + \varepsilon_k)!(d-1)!}{(|\alpha|+d)!} \\
&=& \frac{\alpha!(d-1)!}{(|\alpha|+d-1)!} \bigg(1 + \lambda |\alpha| +  \frac{|\alpha|-1}{|\alpha|+d}\sum_{j=1}^d c_k \alpha_k \bigg).
\eeqn  
Set $L(k)=1 + \lambda k,$ $k \in \mathbb Z_+$ and $K_c (\alpha) = \frac{|\alpha|-1}{|\alpha|+d}\sum_{j=1}^d c_k \alpha_k,$ $\alpha \in \mathbb Z^d_+.$   Then the weights $\mathbf w^{(j)}_{\alpha},$ $j=1, \ldots, d,$ $\alpha \in \mathbb Z^d_+,$ of $\mathscr M_z$ are given by
\beqn
 \frac{\|z^{\alpha + \epsilon_j}\|^2_{\mathscr D(\mu_{\lambda, c})}}{\|z^{\alpha}\|^2_{\mathscr D(\mu_{\lambda, c})}} = {\frac{\alpha_j+1}{|\alpha|+d}} \cdot {\frac{L(|\alpha|+1) +  K_c (\alpha + \varepsilon_j) }{L(|\alpha|) +  K_c (\alpha) }}.
\eeqn
Furthermore, by Proposition \ref{prop-rotation-invariant}, $\mathscr M_z$ is a joint $2$-isometry. 
 Finally, we note that in case $d=2,$ the choices $\lambda=1$ and $c=0,$ the Dirichlet-type space $\mathscr D(\mu_{\lambda, c})$ is nothing but the Drury-Arveson space $H^2_2.$
 \eop
\end{example}
%\begin{remark}
%Notice that $\mathscr M_z$ is spherically balanced if and only if $c=0.$
%By \cite[Theorem 2.1]{CY}, $\mathscr M_z$ is spherical if and only if $c=0.$
%\end{remark}

%\section{A family of non-rotation-invariant measures}

Following \cite[Section 1]{AD}, we say that a measure $\mu \in M_+(\partial \mathbb B^d)$ is {\it pluriharmonic} if $P[\mu]$ is a pluriharmonic function in the open unit ball $\mathbb B^d.$
We now exhibit Dirichlet-type spaces associated with a family of pluriharmonic measures (see \cite[Pg 44]{Ru0}).
%By a {\it weighted pluriharmonic measure}, we understand the weighted surface area measure with weight being a function pluriharmonic function.

\begin{proposition} \label{pluri}
For $h$ belonging to the ball algebra $A(\mathbb B^d)$ such that $\mathsf w(z) :=\Re(h(z)) \geqslant 0$ for every $z \in \partial \mathbb B^d,$ let $\mu_\mathsf w$ be the weighted surface area measure with weight function $\mathsf w|_{\partial \mathbb B^d}.$ 
%Let $\mu$ be a pluriharmonic measure on $\partial \mathbb B^d.$
Then
the $d$-tuple $\mathscr M_z$ on $\mathscr D(\mu_{\mathsf w})$ defines a joint $2$-isometry.
\end{proposition}
\begin{proof}
This is immediate from Corollary \ref{Coro-multi}.
\end{proof}

We exhibit below Dirichlet-type spaces associated with a family of pluriharmonic measures, which are not necessarily rotation-invariant.
\begin{example}
Let $\lambda \in \mathbb R$ and $b=(b_1, \ldots, b_d) \in \mathbb R^d$ be such that $\lambda^2 > 2\sum_{j=1}^d |b_j|^2.$  
Consider the weighted surface area measure $\mu_{b, \lambda}$ on $\partial \mathbb B^d$
with weight function $w_{b, \lambda}$ given by
\beqn
w_{b, \lambda}(z) := \lambda + 2\sum_{j=1}^d \Re(b_j z_j), \quad z \in  \mathbb C^d.
\eeqn 
Since $w_{b, \lambda}$ is positive on $\partial \mathbb B^d,$ we may consider the Dirichlet-type space $\mathscr D(\mu_{b, \lambda})$ associated with $\mu_{b, \lambda}.$ Further, since $\mu_{b, \lambda}$ is a pluriharmonic measure, 
by Proposition \ref{pluri}, the $d$-tuple $\mathscr M_z$ on $\mathscr D(\mu_{b, \lambda})$ is a joint $2$-isometry.
We claim that
\beq 
\label{i-p-non-r-1}
\inp{z^{\alpha}}{z^{\beta}}_{\circ} = \begin{cases} \lambda \frac{|\alpha|\alpha!(d-1)!}{(|\alpha|+d-1)!}  & \mbox{if~} \alpha = \beta, \vspace{.1cm}\\ 
\frac{|\alpha|(\alpha +\varepsilon_l)!(d-1)!}{(|\alpha|+d)!} \, b_l  & \mbox{if}~ \alpha + \varepsilon_l=\beta, ~l=1, \ldots, d,\vspace{.1cm} \\ 
\frac{(|\alpha|-1)\alpha!(d-1)!}{(|\alpha|+d-1)!}  \,\overline{b}_l  & \mbox{if}~ \beta + \varepsilon_l=\alpha, ~l=1, \ldots, d, \vspace{.1cm}\\
0 & \mbox{otherwise}.
\end{cases}
\eeq
To see this, let $\alpha, \beta \in \mathbb Z^d_+$ and let $\delta_{\alpha, \beta}$ denote the Kronecker delta function. Note that for $l=1, \ldots, d,$ by \eqref{Diri-formula}, 
\beqn
\int_{\mathbb B^d}  \inp{\nabla z^{\alpha}}{\nabla z^{\beta}}  z_l dV(z) &=&
 \delta_{\alpha + \varepsilon_l, \beta}  \frac{d!}{(|\alpha|+d)!}   \sum_{k=1}^d \alpha_k \beta_k   (\alpha + \varepsilon_l -\varepsilon_k)! 
 \\ 
 &=& \delta_{\alpha + \varepsilon_l, \beta}  \frac{|\alpha|(\alpha + \varepsilon_l)!d!}{(|\alpha|+d)!}. 
  %\int_{\mathbb B^d}   z^{\alpha+ \varepsilon_l -\varepsilon_k} 
 %\overline{z}^{\beta - \varepsilon_k} dV(z)
\eeqn
Similarly, one may see that
\beqn
\int_{\mathbb B^d}  \inp{\nabla z^{\alpha}}{\nabla z^{\beta}}  \overline{z}_l dV(z)
 = \delta_{\alpha, \beta + \varepsilon_l}  \frac{(|\alpha|-1)\alpha! d!}{(|\alpha|+d-1)!}, \quad l=1, \ldots, d.
  %\int_{\mathbb B^d}   z^{\alpha+ \varepsilon_l -\varepsilon_k} 
 %\overline{z}^{\beta - \varepsilon_k} dV(z)
\eeqn
Combining these two identities with \eqref{Diri-formula}, we obtain \eqref{i-p-non-r-1}. 
It follows that $\mathscr M_z$ is a joint $2$-isometry that is not a weighted multishift, unless $b=0.$
\eop
\end{example}

%{\blue We conclude the section with a family of joint isometries arising from Dirichlet-type spaces associated with Dirac delta measures.
%\begin{example}
%Let $\zeta_0 \in \partial \mathbb B^d \setminus (\mathbb C \setminus \{0\})^d$ and let $\mu_{\zeta_0}$ denote the Dirac delta measure with point mass at $\zeta_0.$ Note that 
%\beqn
%\int_{\partial \mathbb B^d} \zeta^{\alpha} \overline{\zeta}^{\beta} d\mu_{\zeta_0}(\zeta) = \zeta^{\alpha}_0 \overline{\zeta_0}^{\beta}, \quad \alpha, \beta \in \mathbb Z^d_+.
%\eeqn
%Thus, if $\mathcal M=\bigvee\{z^{\alpha} : \zeta^{\alpha}_0\overline{\zeta_0}^{\beta}=0\},$ then for every polynomial $p \in \mathcal M,$ by Theorem \ref{Richter-several}, $ \sum_{|\gamma|=k} \frac{|\gamma|!}{\gamma!}  \|z^{\gamma}p\|^2_{\mathscr D(\mu_{\zeta_0})}$ is a constant polynomial in $k$. In particular, for $k=1, \ldots, d,$ $z_k$ is a multiplier for $\mathscr D(\mu_{\zeta_0}),$ $\mathcal M$ is invariant under the $d$-tuple $\mathscr M_z$ and $\mathscr M_z|_{\mathcal M}$  defines a joint isometry.
%\eop
%\end{example}
%}

\section{A complex moment problem and a Gramian matrix associated with joint $m$-isometries}

%\section{A complex moment problem in several variables}

In the first half of this section, we discuss a complex moment problem in several variables
that arises naturally in the study of joint m-isometries. We have already encountered such a problem in Theorem \ref{Richter-several}. 
%It turns out that a
%vector-valued kernel function on a $d$-fold Cartesian product of
%non-negative integers is a complex moment sequence if and only if it satisfies a 
%spherical Toeplitz-type condition (see \eqref{s-Toeplitz} below).  
The solution to this moment problem can be considered as a spherical analog of the solution of the trigonometric moment problem. In the scalar one-dimensional case, this is commonly attributed to Akhiezer and Krein (see \cite[Theorem 1.4]{ST}). In the general case, this can be derived from \cite[Proposition 34]{SS} and \cite[Proposition 2.1]{St}.  Although this result has been known to experts in the moment theory (refer to \cite{F, St, SS, CSS}), we could not locate it in the form we need in this paper. Hence we include the statement and relegate its proof to the appendix. 
\begin{theorem} \label{CMP}
For a complex Hilbert space $\mathcal M,$ let
$\phi : \mathbb Z^d_+ \times \mathbb Z^d_+ \rar \mathcal B(\mathcal M)$ be a $\mathcal B(\mathcal M)$-valued kernel function. Then the following are equivalent$:$
\begin{enumerate}
\item[$(i)$] There exists $F \in M_+(\partial \mathbb B^d, \mathcal B(\mathcal M))$ such that 
\beq \label{TMP}
\phi(\alpha, \beta) =  \int_{\partial \mathbb B^d} \zeta^{\alpha} \overline{\zeta}^{\beta} dF(\zeta), \quad \alpha, \beta \in \mathbb Z^d_+.
\eeq
\item[$(ii)$] The kernel function $\phi$ is positive definite and 
\beq  \label{s-Toeplitz} \sum_{j=1}^d \phi(\alpha + \varepsilon_j, \beta + \varepsilon_j) =\phi(\alpha, \beta), \quad  \alpha, \beta \in \mathbb Z^d_+. \eeq
\end{enumerate} 
The semispectral measure $F$ is uniquely determined by \eqref{TMP}.
%If \eqref{TMP} holds, then $\phi$ is determinate.
\end{theorem}

Here is an application of Theorem \ref{TMP} to joint $m$-isometries.
\begin{corollary} 
%\label{CMP-m-iso}
Let $T$ be a joint $m$-isometric $d$-tuple on $\mathcal H.$ Then there exists $F \in M_+(\partial \mathbb B^d, \mathcal B(\mathcal H))$ such that
\beqn
\sum_{j=0}^{m-1} (-1)^{j+m-1} \binom{m-1}{j}  T^{*\beta}Q^j_T(I)T^{\alpha} = \int_{\partial \mathbb B^d}\zeta^{\alpha}\overline{\zeta}^{\beta}dF(\zeta), \quad \alpha, \beta  \in \mathbb Z^d_+.
\eeqn
\end{corollary}
\begin{proof} 
Define a $\mathcal B(\mathcal H)$-valued kernel function $\phi : \mathbb Z^d_+ \times \mathbb Z^d_+ \rar \mathcal B(\mathcal H)$ by 
\beqn
\phi(\alpha, \beta) = \sum_{j=0}^{m-1} (-1)^{j+m-1}\binom{m-1}{j}  T^{*\beta}Q^j_T(I)T^{\alpha}, \quad \alpha, \beta  \in \mathbb Z^d_+.
\eeqn
By \cite[Proposition 2.3]{GR}, $(-1)^{m-1}B_{m-1}(T) \geqslant 0$ (see \eqref{BMQT}), and hence \beq \label{id-m-iso} \phi(\alpha, \beta)=T^{*\beta}((-1)^{m-1}B_{m-1}(T))T^{\alpha} \eeq is positive definite.
In view of Theorem \ref{CMP}, it now suffices to check that $\phi$ satisfies \eqref{s-Toeplitz}. Since $T$ is an $m$-isometry, it follows from the identity $B_{m}(T)=B_{m-1}(T)-Q_T(B_{m-1}(T))$ that 
$B_{m-1}(T)=Q_T(B_{m-1}(T))$ (see \cite[Equation (2.1)]{GR}). Hence, by \eqref{id-m-iso}, for every $\alpha, \beta  \in \mathbb Z^d_+,$
\beqn
\sum_{j=1}^d \phi(\alpha + \varepsilon_j, \beta + \varepsilon_j)  &=& \sum_{j=1}^d (-1)^{m-1} T^{*\beta + \varepsilon_j}B_{m-1}(T)T^{\alpha + \varepsilon_j} \\ &=&  (-1)^{m-1} T^{*\beta} Q_T(B_{m-1}(T))T^{\alpha} \\
&=& (-1)^{m-1} T^{*\beta} B_{m-1}(T)T^{\alpha} \\
&=& \phi(\alpha, \beta).
\eeqn
This completes the proof.
\end{proof}
In case the joint $2$-isometric $d$-tuple $T$ admits the wandering subspace property, the semispectral measure $F$ as ensured in the above corollary may be replaced by a $\mathcal B(\ker T^*)$-valued semispectral measure, where $\ker T^*$ denotes the joint kernel $\cap_{j=1}^d \ker T^*_j$ of $T^*.$
As we will see in the remaining part of this section, this leads to a model theorem for the spherical moments of joint
$m$-isometries that admit the wandering subspace property.
%In the context of last result, a natural question arises whether it is possible to 

Following \cite{Sh}, we define a $d$-tuple $\sigma$ of {\it matrix backward shifts} $\sigma_1, \ldots, \sigma_d$ on the matrix array $\mathscr A=[\![\mathscr A_{\alpha, \beta}]\!]_{\alpha, \beta \in \mathbb Z^d_+}$ as follows:
\beqn
\sigma_j \mathscr A  := \big[\!\big[ \mathscr A_{\alpha + \varepsilon_j, \beta + \varepsilon_j}\big]\!\big]_{\alpha, \beta \in \mathbb Z^d_+}, \quad j=1, \ldots, d.
\eeqn 
Since the actions of $\sigma_1, \ldots, \sigma_d$ are mutually commuting, we get the following:
\beq \label{sigma-gamma}
\sigma^\gamma \mathscr A  = \big[\!\big[ \mathscr A_{\alpha + \gamma, \beta + \gamma}\big]\!\big]_{\alpha, \beta \in \mathbb Z^d_+}, \quad \gamma \in \mathbb Z^d_+.
\eeq
For an integer $n \in \mathbb Z_+,$ let $\triangle_{\mathscr A, n}$ denote the matrix given by
\beq
\label{defect-Gram}
\triangle_{\mathscr A, n}:=\sum_{j=0}^n (-1)^{j+n} \binom{n}{j} \sum_{|\gamma|=j} \frac{|\gamma|!}{\gamma!} \, \sigma^\gamma \mathscr A.
\eeq 
To state the main result of this section, we need another notion. A commuting $d$-tuple $T$ on $\mathcal H$ admits the {\it wandering subspace property} if $$\mathcal H = \bigvee \{T^{\alpha}h : h \in \ker T^*, \alpha \in \mathbb Z^d_+\}.$$  Note that the joint kernel $\ker T^*$ of $T^*$ is a {\it wandering subspace} in the sense that $T^{\alpha}(\ker T^*)$ is orthogonal to $\ker T^*$ for every non-zero $\alpha \in \mathbb Z^d_+.$

We are now ready to state the main result of this section.
\begin{theorem} \label{main-thm}
Let $\mathcal H$ be a separable Hilbert space and let $T$ be a commuting $d$-tuple on $\mathcal H.$ Let $\{f_{j}\}_{j \in I}$ be an orthonormal basis for the joint kernel of $T^*$ for some nonempty directed set $I.$  Consider the Gramian matrix $\mathscr G=[\![\mathscr G_{\alpha, \beta}]\!]_{\alpha, \beta \in \mathbb Z^d_+}$ associated with $T$ given by
\beqn
\mathscr G_{\alpha, \beta} := \big[\!\big[ \inp{T^{\beta}f_j}{T^{\alpha}f_i}\big]\!\big]_{i, j \in I}, \quad \alpha, \beta \in \mathbb Z^d_+,
\eeqn
and let $\triangle_{\mathscr G, j}$ be as defined in \eqref{defect-Gram} $($with $\mathscr A$ replaced by $\mathscr G).$
Assume that $T$ admits the wandering subspace property. 
%Then $T$ is a joint $m$-isometry if and only if 
% \beq \label{m-iso-Gram} \sum_{|\gamma|=k}  \frac{|\gamma|!}{\gamma!} \, \sigma^\gamma \mathscr G = \sum_{j=0}^{m-1} {k \choose j}\triangle_{\mathscr G, j}, \quad k \in \mathbb Z_+.\eeq
Then the following are equivalent$:$
\begin{enumerate}
\item[$(i)$] $T$ is a joint $m$-isometry, 
\item[$(ii)$] $\triangle_{\mathscr G, m}=0$,
\item[$(iii)$] $\sum_{|\gamma|=k}  \frac{|\gamma|!}{\gamma!} \, \sigma^\gamma \mathscr G = \sum_{j=0}^{m-1} \binom{k}{j} \triangle_{\mathscr G, j}$ for every $k \in \mathbb Z_+,$ where we used the convention that $\binom{k}{j}=0$ if $k < j,$
\item[$(iv)$] $\sum_{j=1}^d \sigma_j \triangle_{\mathscr G, m-1} = \triangle_{\mathscr G, m-1}.$
%if $\triangle_{\mathscr G, m-1}=[\![\mathscr B_{\alpha, \beta}]\!]_{\alpha, \beta \in \mathbb Z^d_+},$ then 
%\beq \label{Toeplitz-Gram-1}
%\sum_{j=1}^d \mathscr B_{\alpha + \varepsilon_j, \beta + \varepsilon_j} = \mathscr B_{\alpha, \beta}, \quad \alpha, \beta \in \mathbb Z^d_+.
%\eeq
\end{enumerate}
 Moreover, if $(i)$ holds, then $\triangle_{\mathscr G, m-1} \geqslant 0,$ and there exists a semi-spectral measure $F \in M_+(\partial \mathbb B^d, \mathcal B(\ell^2(I)))$ such that for every $\alpha, \beta \in \mathbb Z^d_+,$
 \beq \label{measure-F}
\sum_{j=0}^{m-1} (-1)^{j+m-1} \binom{m-1}{j}  \sum_{|\gamma|=j} \frac{|\gamma|!}{\gamma!} \, \mathscr G_{\alpha + \gamma, \beta + \gamma} =
 \int_{\partial \mathbb B^d} \zeta^{\alpha} \overline{\zeta}^{\beta} dF(\zeta).
 \eeq
\end{theorem}
%{\blue \begin{remark} Let us discuss the case $m=1$ of Theorem \ref{main-thm}.
%Note that $T$ is a joint isometry if and only if $\sum_{|\gamma|=k}  \frac{|\gamma|!}{\gamma!} \, \sigma^\gamma \mathscr G = \mathscr G$ for every $k \in \mathbb Z_+.$ In this case, there exists a $\mathcal B(\ell^2(I))$-valued semispectral measure $F$ on the unit sphere $\partial \mathbb B^d$ in $\mathbb C^d$ such that \beqn
% \mathscr G_{\alpha, \beta} = \int_{\partial \mathbb B^d} \!\!\! z^{\alpha} \overline{z}^{\beta} dF(z), \quad \alpha, \beta \in \mathbb Z^d_+.
% \eeqn
%\end{remark}
%}

%\section{Proof of Theorem \ref{main-thm}}
In the proof of Theorem \ref{main-thm}, we need a lemma that relates certain operator identities to the respective matrix identities involving the Gramian matrix (cf. \cite[Theorem 2.15]{Sh}). 
\begin{lemma} \label{main-lem-1}
Let $T, \mathscr G, \triangle_{\mathscr G, m}$ be as in the statement of Theorem \ref{main-thm}.
%Let $T$ be a commuting $d$-tuple of operators in $\mathcal B(\mathcal H)$ and let $\{f_{j}\}_{j \in I}$ be an orthonormal basis for the joint kernel of $T^*,$ where $I$ is a directed set. Consider the Gramian block-matrix $\mathscr G=[\![\mathscr G_{\alpha, \beta}]\!]_{\alpha, \beta \in \mathbb Z^d_+}$ associated with the vectors $T^{\alpha}f_j$:
%\beq \label{Gram}
%\mathscr G_{\alpha, \beta} := \big[\!\big[ \inp{T^{\beta}f_j}{T^{\alpha}f_i}\big]\!\big]_{i, j \in I}, \quad \alpha, \beta \in \mathbb Z^d_+,
%\eeq
%and for a positive integer $m,$ let $\triangle_{\mathscr G, m}$ denote the block-matrix given by
%\beq
%\label{Gram-defect}
%\triangle_{\mathscr G, m}:=\sum_{j=0}^m (-1)^{j+m} {m \choose j} \sum_{|\gamma|=j} \frac{|\gamma|!}{\gamma!} \, \sigma^\gamma \mathscr G.
%\eeq 
%\beqn
%\mathscr G_{\alpha, \beta} := \big[\!\big[ \inp{v_{\alpha, j}}{v_{\beta, i}}\big]\!\big]_{i, j \in I}
%\eeqn
If $T$ admits the wandering subspace property, then the following are true$:$
\begin{enumerate}
\item[$(i)$] $T$ is a joint $m$-concave $d$-tuple if and only if $\triangle_{\mathscr G, m} \leqslant 0$,
\item[$(ii)$] $T$ is a joint $m$-convex $d$-tuple if and only if  $\triangle_{\mathscr G, m} \geqslant 0.$
\item[$(iii)$] 
$T$ is a joint $m$-isometry if and only if \beq \label{char-interm} \sum_{|\gamma|=k}  \frac{|\gamma|!}{\gamma!} \, \sigma^\gamma \mathscr G = \sum_{j=0}^{m-1} 
\binom{k}{j} \triangle_{\mathscr G, j}, \quad k \in \mathbb Z_+, \eeq
\item[$(iv)$] $T$ is a joint $m$-isometry if and only if 
%$\triangle_{\mathscr G, m}=0.$
$$\sum_{j=1}^d \sigma_j \triangle_{\mathscr G, m-1} = \triangle_{\mathscr G, m-1}.$$
\end{enumerate}
\end{lemma}
\begin{proof}
%(i)$\Rightarrow$(ii): 
Assume that $T$ admits the wandering subspace property. Since the verifications of (i) and (ii) are similar, we only verify (i).
For finite subsets $A \subseteq I$ and $B \subseteq \mathbb Z^d_+$, set \beq \label{finM} f:=\sum_{j \in A} \sum_{\beta \in B} c_{j, \beta}(f) T^{\beta}f_j, \quad C_{\beta}(f)=(c_{j, \beta}(f))_{_{j \in A}} \in \ell^2(A),~\beta \in B, \eeq
and let $\mathcal N$ denote the subspace of $\mathcal H$ consisting of all such vectors $f$ in $\mathcal H$. Note that the wandering subspace property for $T$ is equivalent to the density of the subspace $\mathcal N$ in $\mathcal H.$ Fix $f \in \mathcal N$ as given in \eqref{finM} and observe that
\beqn 
\|f\|^2 = \|\sum_{j \in A} \sum_{\beta \in B} c_{j, \beta}(f) T^{\beta}f_j\|^2 = \sum_{\alpha, \beta \in B} \inp{\mathscr G_{\alpha, \beta}\,C_{\beta}(f)}{C_{\alpha}(f)}.
\eeqn
Letting $C(f)=(C_{\alpha}(f))_{\alpha \in B},$ we obtain 
\beq \label{mono}
\|T^{\gamma}f\|^2 &=& \notag \sum_{\alpha, \beta \in B} \inp{\mathscr G_{\alpha + \gamma, \beta + \gamma}\,C_{\beta}(f)}{C_{\alpha}(f)} \\ & \overset{\eqref{sigma-gamma}} = & \inp{(\sigma^\gamma \mathscr G) C(f)}{C(f)}, \quad \gamma \in \mathbb Z^d_+.
\eeq
Thus, by \eqref{QT-n}, for any positive integer $k,$ we have
\beq \label{gen-tuple}
\inp{Q^k_T(I)f}{f}  =  \sum_{|\gamma|=k} \frac{|\gamma|!}{\gamma!} \, \|T^{\gamma}f\|^2 = \sum_{|\gamma|=k}  \frac{|\gamma|!}{\gamma!} \, \inp{(\sigma^\gamma \mathscr G) C(f)}{C(f)}.
\eeq
%Let $n$ be a positive integer. 
Since $f$ is varying over the dense subspace $\mathcal N$ of $\mathcal H$, we deduce from \eqref{defect-Gram} that $T$ is a joint $m$-concave $d$-tuple if and only if
\beq \label{latter one}
 \left \langle {\triangle_{\mathscr G, m} C(f)},{C(f)} \right \rangle \leqslant 0, \quad f \in \mathcal N. 
\eeq
Moreover, if we vary $f$ over $\mathcal N,$ $C(f)$ varies over scalar column vectors of all sizes.
Since $\sigma^\gamma \mathscr G$ is a self-adjoint matrix, 
\eqref{latter one} is equivalent to $\triangle_{\mathscr G, m} \leqslant 0.$ 
This completes the verification of (i). 
%Clearly, (iii) follows from (i) and (ii).

To see (iii), 
recall the fact that any joint $m$-isometry $d$-tuple $T$ satisfies the following operator identity:
\beq \label{gen-poly}
Q^k_T(I) = \sum_{j=0}^{m-1} (-1)^j \binom{k}{j}  B_j(T), \quad k \in \mathbb Z_+.
\eeq
This is a consequence of \cite[Lemma 2.2]{GR} and the fact that for any joint $m$-isometry $T,$ $B_j(T)=0$ for every $j \geqslant m.$ Let $k$ be a positive integer. 
Note that by \eqref{gen-tuple} and the density of $\mathcal N$ in $\mathcal H$, \eqref{gen-poly} holds if and only if 
\beqn
\sum_{j=0}^{m-1} (-1)^j \binom{k}{j}  \inp{ B_j(T)f}{f}  = \sum_{|\gamma|=k}  \frac{|\gamma|!}{\gamma!} \, \inp{(\sigma^\gamma \mathscr G) C(f)}{C(f)}, \quad f \in \mathcal N,
\eeqn
which, in view of \eqref{BMQT}, \eqref{mono} and \eqref{defect-Gram}, is equivalent to 
\beqn
%\sum_{j=0}^{m-1} {k \choose j} \left(\sum_{l=0}^j (-1)^{l+j} {j \choose l} \sum_{|\gamma|=l}\frac{|\gamma||}{\gamma!} \inp{(\sigma^\gamma \mathscr G) C}{C},\right)  
\sum_{j=0}^{m-1} \binom{k}{j} \inp{\triangle_{\mathscr G, j}C(f)}{C(f)} 
= \sum_{|\gamma|=k}  \frac{|\gamma|!}{\gamma!} \, \inp{(\sigma^\gamma \mathscr G) C(f)}{C(f)}, \quad f \in \mathcal N.
\eeqn
This is, as in the last paragraph, seen to be equivalent to \eqref{char-interm}. 
The necessary part in (iii) now follows from \eqref{gen-poly}.
To see the sufficiency part, 
in view of the discussion above, we may assume that \eqref{gen-poly} holds. In particular, we have the polynomial $p_{x, y}(k) = \inp{Q^k_T(I)x}{y}$ of degree at most $m-1$ for every $x,  y \in \mathcal H.$
However, for any complex polynomial $p$ in one variable of degree at most $m-1,$ 
$\sum_{j=0}^m (-1)^j \binom{m}{j} p(j) = 0$
 (see \cite[Proposition 2.1]{JJS}).
We may now apply the above fact to each $p_{x, y}$ to get (iii). 

Finally, to see (iv), note that by \cite[Equation (2.1)]{GR}, $T$ is a joint $m$-isometry if and only if $B_{m-1}(T)=Q_T(B_{m-1}(T)),$ and
argue as above.
\end{proof}
%\begin{remark}
It is worth noting that the idea of Gramian may be employed to give an alternate proof of \cite[Theorem 2.15]{Sh}. We leave the details to the reader.

\begin{proof}[Proof of Theorem \ref{main-thm}] 
The equivalence of (i)-(iv) is immediate from Lemma \ref{main-lem-1}. 
%To see (i)$\Rightarrow$(iii), note that by \cite[Lemma 2.2]{GR} that any $m$-isometry $T$ satisfies the following operator identity:
%\beq \label{gen-poly}
%Q^k_T(I) = \sum_{j=0}^{m-1} (-1)^j {k \choose j}  B_j(T), \quad k \in \mathbb Z_+,
%\eeq
%and apply Lemma \ref{main-lem-1}(iii).
%To see 
%(iii)$\Rightarrow$(i), in view of Lemma \ref{main-lem-1}(iii), we may assume that \eqref{gen-poly} holds. In particular, we have the polynomial $p_{x, y}(k) = \inp{Q^k_T(I)x}{y}$ of degree at most $m-1$ for every $x,  y \in \mathcal H.$
%However, for any complex polynomial $p$ in one variable of degree at most $m-1,$ 
% \beqn \sum_{j=0}^m (-1)^j {m \choose j} p(j) = 0 \eeqn
% (see \cite[Proposition 2.1]{JJS}).
%We may now apply the above fact to each $p_{x, y}$ to get (i). The equivalence (i) $\Leftrightarrow$ (iv) follows from  
%Finally, to see  \eqref{Toeplitz-Gram-1}, note that 
%$\triangle_{\mathscr G, m}=\sum_{j=1}^d \sigma_j \triangle_{\mathscr G, m-1} - \triangle_{\mathscr G, m-1}$ (cf.). 
To see the remaining part, assume that $T$ is a joint $m$-isometric $d$-tuple.  By \cite[Proposition 2.3]{GR}, $\triangle_{T, m-1} \geqslant 0$, and hence by Lemma \ref{main-lem-1}(ii), \beq \label{Toeplitz-Gram-2} \triangle_{\mathscr G, m-1} \geqslant 0. \eeq 
%\beq \label{op-gramian}
%\mbox{\eqref{positivity}} \Longleftrightarrow \quad \triangle_{\mathscr G, n}:=\sum_{j=0}^n (-1)^{j+n} {n \choose j} \sum_{|\gamma|=j} \frac{|\gamma|!}{\gamma!} \, \sigma^\gamma \mathscr G \geqslant 0 \quad \mbox{resp.}~\leqslant 0.  
%\eeq
%The first equivalence follows from Lemma \ref{main-lem-2}. 
To see the remaining part, set $$\phi(\alpha, \beta) := \sum_{j=0}^{m-1} (-1)^{j+m-1} \binom{m-1}{j}  \sum_{|\gamma|=j} \frac{|\gamma|!}{\gamma!} \, \mathscr G_{\alpha + \gamma, \beta + \gamma}, \quad \alpha, \beta \in \mathbb Z^d_+,$$ and
note that by \eqref{Toeplitz-Gram-2} and Lemma \ref{main-lem-1}(iv) together with \eqref{defect-Gram}, $\phi$
is positive definite and satisfies \eqref{s-Toeplitz}. 
Further, a simple application of the Cauchy-Schwarz inequality shows that for every $\alpha, \beta \in \mathbb Z^d_+$,
\beqn
|\inp{\mathscr G_{\alpha, \beta} X}{Y}_{\ell^2(I)}| \leqslant \|T^{\alpha}\| 
\|T^{\beta}\| \|X\|_{\ell^2(I)} \|Y\|_{\ell^2(I)}, \quad X, Y \in \ell^2(I).
\eeqn
This shows that $\phi(\alpha, \beta) \in \mathcal B(\ell^2(I))$ for every $\alpha, \beta \in \mathbb Z^d_+.$
Now applying Theorem \ref{TMP} (with $\mathcal M=\ell^2(I)$) completes the proof.
\end{proof}

Before we see an analytic model for the spherical moments of joint $2$-isometries, it is worth noting that analytic cyclic joint $2$-isometries can not be modeled as the multiplication tuples on a Dirichlet-type space in dimension bigger than $1.$
\begin{example} \label{a-model-fails} Let $\nu$ be a positive finite Borel measure on the unit circle $\mathbb T$ and let $\mathscr M_z$ be the operator of the multiplication by $z$ acting on the Dirichlet-type space $\mathscr D(\nu)$. By \cite[Theorem 3.6, 3.7 and Corollary 3.8]{Ri}, $\mathscr M_z$ defines an analytic cyclic $2$-isometry.   
Consider the commuting pair $T=(T_1, T_2),$ where $T_j=\mathscr M_z/\sqrt{2}$ for $j=1, 2.$ Since $\mathscr M_z$ is an analytic cyclic operator with cyclic vector $1$, so is $T.$ Also, since $\mathscr M_z$ is a $2$-isometry, by \cite[Proposition 3.7]{CS}, $T$ is a joint $2$-isometry. Since $z_1 \neq z_2$ on $\mathbb B^d,$ there is no $\mu \in M_+(\partial \mathbb B^2)$  and a unitary $U : \mathscr D \rar \mathscr D(\mu)$ such that $U1=1$ and $\mathscr M_{z_j}U=UT_j$ for $j=1, 2.$
%Let $D$ be the Dirichlet shift acting on the classical Dirichlet space $\mathscr D$, that is, $D$ is the weighted shift with weights $\big\{\sqrt{\frac{n+2}{n+1}}\big\}_{n  \in \mathbb Z_+}$ with respect to the orthonormal basis $\{z^n/\|z^n\|\}_{n \in \mathbb Z_+}.$ 
%Consider the commuting pair $T=(T_1, T_2),$ where $T_j=D/\sqrt{2}$ for $j=1, 2.$ Since $D$ is an analytic cyclic operator with cyclic vector $1$, so is $T.$ Also, since $D$ is a $2$-isometry, by \cite[Proposition 3.7]{CS}, $T$ is a joint $2$-isometry. Since $z_1 \neq z_2$ on $\mathbb B^d,$ there is no positive finite Borel measure on $\partial \mathbb B^2$ and a unitary $U : \mathscr D \rar \mathscr D(\mu)$ such that $U1=1$ and $\mathscr M_{z_j}U=UT_j$ for $j=1, 2.$
\eop
\end{example}

We conclude the paper with a model theorem for the spherical moments of joint $2$-isometries (cf. \cite[Pg 30]{Ri-3}).
\begin{corollary} \label{a-model}
Let $T$ be a commuting $d$-tuple on a separable Hilbert space $\mathcal H$ and let $\{f_{j}\}_{j \in I}$ be an orthonormal basis for the joint kernel $\ker T^*$ of $T^*$ for some nonempty directed set $I.$   
%If $T$ admits the wandering subspace property,  then 
If $T$ is a joint $2$-isometry, then there exists $F_T \in M_+(\partial \mathbb B^d, \mathcal B(\ell^2(I)))$ such that  for every $\alpha, \beta \in \mathbb Z^d_+,$ $k \in \mathbb Z_+$ and $x, y \in \ker T^*,$
\beq \label{s-moments}
&& \notag \inp{Q^k_T(I)T^{\alpha }x}{T^{\beta}y} - \inp{T^{\alpha }x}{T^{\beta}y}  \\ &=& \sum_{|\gamma|=k} \frac{|\gamma|!}{\gamma!}  \inp{z^{\alpha + \gamma}x}{z^{\beta + \gamma}y}_{\mathscr D(F_T)} - \inp{z^{\alpha}x}{z^{\beta}y}_{\mathscr D(F_T)}.
\eeq
%Moreover, if $z_j,$ $j=1, \ldots, d$ a multiplier of the Dirichlet-type space $\mathscr D(F_T),$ then 
\end{corollary}
\begin{proof} Assume that $T$ is a joint $2$-isometry. Since the restriction of a joint $2$-isometry to an invariant subspace is again a joint $2$-isometry, after replacing $\mathcal H$ by the invariant subspace $\bigvee \{T^{\alpha}h : h \in \ker T^*, \alpha \in \mathbb Z^d_+\}$ of $T,$ if necessary, we may assume that $T$ admits the wandering subspace property.
It suffices to check the above formula for $x, y \in \{f_{j}\}_{j \in I}.$ The existence of the semispectral measure $F_T$ is ensured by Theorem \ref{main-thm} (with $m=2$). A simple calculation using \eqref{measure-F} shows that
 \beq \label{F-T-proof}
    \big[\!\big[ \inp{(Q_T(I)-I)T^{\beta}f_j}{T^{\alpha}f_i}\big]\!\big]_{i, j \in I} 
 = \int_{\partial \mathbb B^d} \zeta^{\alpha} \overline{\zeta}^{\beta} dF(\zeta), \quad \alpha, \beta \in \mathbb Z^d_+.
 \eeq
One may now consider the Dirichlet-type space $\mathscr D(F_T)$ associated with $F_T.$ Since $Q^k_T(I)=I+ k(Q_T(I)-I),$ $k \geqslant 1,$ the desired conclusion now follows from Theorem \ref{Richter-several} and \eqref{F-T-proof}. 
%Let $\alpha, \beta \in \mathbb Z^d_+$ such that $\alpha_j\beta_j=0,$ $j=1, \ldots, d.$ 
%Hence, by Theorem \ref{Richter-several}, for any $k \in \mathbb Z_+$ and $x, y \in \mathcal M,$
%\beqn
%&& \sum_{|\gamma|=k} \frac{|\gamma|!}{\gamma!}  \inp{z^{\alpha+ \gamma}x}{z^{\beta+ \gamma}y}_{\mathscr D(F)} \\ &=& \int_{\partial \mathbb B^d} \int_{\mathbb B^d} \sum_{|\gamma|=k} \frac{|\gamma|!}{\gamma!} \sum_{j=1}^d \inp{P[F](z) \frac{\partial z^{\alpha + \gamma}}{\partial z_j}x}{\frac{\partial z^{\beta + \gamma}}{\partial z_j}y} dV(z)  \\ &=& k \int_{\partial \mathbb B^d} \zeta^{\alpha} \overline{\zeta}^{\beta} \inp{dF(\zeta)x}{y} \\
%&=& k \Big( \sum_{j=1}^d \inp{T^{\alpha + \varepsilon_j}x}{T^{\beta + \varepsilon_j}y} - \inp{T^{\alpha}x}{T^{\beta}y} \Big) \\
%&=&  \inp{(Q^k_T(I)-I)T^{\alpha }x}{T^{\beta}y}. 
%\eeqn
%This completes the proof.
%To see the remaining part, assume that $z_j,$ $j=1, \ldots, d$ a multiplier of $\mathscr D(\mu).$ For $p \in \mathbb C[z_1, \ldots, z_d]$ and $x \in \ker T^*,$ set
%$U(p(T)x)=p \otimes x,$ and note that
%\beqn
%\|Up(T)x\|^2 = 
%\eeqn 
\end{proof}
\begin{remark}
%Clearly, the formula \eqref{s-moments} extends to vector-valued polynomials in $T$ and $z.$ 
In case of $d=1,$ it may be concluded from the formula \eqref{s-moments}  that the Gramian matrices associated with $T$ and the multiplication $d$-tuple $\mathscr M_z$ on $\mathscr D(F_T)$ coincide. In general, this fails in case $d \geqslant 2$ (see Example \ref{a-model-fails}). However, for a $2$-variable weighted shift $T,$ it can be seen by induction 
%on one of the coordinates 
that these Gramian matrices are same provided they coincide for indices $\alpha=k\varepsilon_j=\beta,$ $k \in \mathbb Z_+,$ $j=1, 2.$
\end{remark}

As far as the classification of joint $2$-isometries is concerned, the only known case, due to Richter-Sundberg, is of finite dimensional joint $2$-iso-metries (see  \cite[Pg 17]{Ri-3} for the statement and \cite[Theorem 3.2]{Le} for a proof). 
Nevertheless, the results in this paper set the ground for the problem of classifying all analytic joint $2$-isometries that can be modeled as the multiplication tuples $\mathscr M_z$ on a Dirichlet-type space $\mathscr D(F).$ 
%We hope to come back to this problem in the near future.
%appears to be intractable at the moment.

\section*{Appendix: Proof of Theorem \ref{TMP}}

In this appendix, we present a proof of Theorem \ref{TMP} that
exploits the joint subnormality of spherical isometries (see \cite[Proposition 2]{At}). 
\begin{proof}[Proof of Theorem \ref{TMP}] The implication (i)$\Rightarrow$(ii) is a routine verification.

(ii)$\Rightarrow$(i): 
The essential idea of the proof of this implication is due to Fuglede (see \cite[Section 4]{F}). 
Consider the complex vector space ${\mathcal M}[z]:=\mathbb C[z_1, \ldots, z_d] \otimes \mathcal M$ of $\mathcal M$-valued polynomials in $z_1, \ldots, z_d$. For finite subsets $A, B$ of $\mathbb Z^d_+$ and $x_{\alpha}, y_{\beta} \in \mathcal M,$ $\alpha \in A$ and $\beta \in B$, define
\beq \label{ip}
\Big \langle \sum_{\alpha \in A} z^{\alpha} \otimes x_\alpha \, , \,{\sum_{\beta \in B} z^{\beta} \otimes y_\beta}\Big \rangle := \sum_{\alpha \in A} \sum_{\beta \in B} \inp{\phi(\alpha, \beta)x_\alpha}{y_\beta}_{_\mathcal M}.
\eeq
Since $\phi$ is positive definite, this defines a semi-inner-product on  ${\mathcal M}[z]$. 
Consider the quotient space ${\mathcal M}[z]/\mathcal N,$ where $\mathcal N:=\{f \in {\mathcal M}[z] : \inp{f}{f}=0\}.$ 
Let $\mathcal H$ be the completion of ${\mathcal M}[z]/\mathcal N.$
Since $
|\inp{f}{g}| \leqslant \|f\|\|g\|$ for every $f, g \in \mathcal M[z],$ $\inp{f}{g}=0$ provided $f$ or $g$ belongs to $\mathcal N,$ and hence for any $[g]=g+\mathcal N \in {\mathcal M}[z]/\mathcal N,$
\beq
\label{quotient-norm}
\|[g]\|^2_{\mathcal H} = \inf_{f \in \mathcal N}\| g + f\|^2 = \|g\|^2.
%\alpha \in \mathbb Z^d_+, ~x \in \mathcal M.
\eeq 
Thus the inner-product on $\mathcal M[z]$ induces an inner-product on ${\mathcal M}[z]/\mathcal N.$
 Consider the commuting $d$-tuple $\mathscr M_z$ of multiplication operators $\mathscr M_{z_1}, \ldots, \mathscr M_{z_d}$ defined on $\mathcal M[z]$. Note that for any finite subset $A$ of $\mathbb Z^d_+$ and $x_{\alpha} \in \mathcal M,$ $\alpha \in A,$ by \eqref{ip} and \eqref{s-Toeplitz}, we have 
\beqn
\sum_{j=1}^d \Big\|\mathscr M_{z_j} \Big(\sum_{\alpha \in A} z^{\alpha} \otimes x_\alpha \Big) \Big\|^2 & = & \sum_{j=1}^d \Big\| \sum_{\alpha \in A} z^{\alpha + \epsilon_j} \otimes x_\alpha \Big\|^2 \\ & =& \sum_{j=1}^d  \sum_{\alpha, \beta \in A}  \inp{\phi(\alpha + \epsilon_j, \beta + \epsilon_j)x_\alpha}{x_\beta}_{_\mathcal M} \\ 
& =& \sum_{\alpha, \beta \in A}  \inp{\phi(\alpha, \beta)x_\alpha}{x_\beta}_{_\mathcal M} \\
&=& \Big\|\sum_{\alpha \in A} z^{\alpha} \otimes x_\alpha  \Big\|^2.
\eeqn
This together with \eqref{quotient-norm} shows that the $d$-tuple $\mathscr M_z$ induces a joint isometry on $\mathcal H,$ say, $S=(S_1, \ldots, S_d),$ given by 
\beq \label{action}
S_j ([f])= [z_j f], \quad [f] \in \mathcal M[z]/\mathcal N, ~j=1, \ldots, d.
\eeq
By \cite[Proposition 2]{At} and the multivariate Bram's characterization of joint subnormality (\cite[Equation (0)]{A-P}), 
there exists $E \in M_+(\partial \mathbb B^d, \mathcal B(\mathcal H))$ such that 
%$E(\partial \mathbb B^d)=I$ and
\beq
\label{subn}
S^{*\beta} S^{\alpha} = \int_{\partial \mathbb B^d}  \zeta^{\alpha}\overline{\zeta}^\beta dE(\zeta), \quad \alpha, \beta \in \mathbb Z^d_+.
\eeq
%\beq
%\label{subn}
%S^{*\beta} S^{\alpha} = \int_{\partial \mathbb B^d}  z^{\alpha}\overline{z}^\beta dE(z), \quad \alpha, \beta \in \mathbb Z^d_+
%\eeq
%(see \cite[Theorem 1.9, Chapter II]{Co} for the one-variable case). 
To complete the proof, we define a bounded linear map $V : \mathcal M \rar \mathcal H$ by 
\beqn V(x)=[1 \otimes x], \quad x \in \mathcal M, \eeqn 
and let $F(\cdot)=V^*E(\cdot)V.$ Clearly, $F$ defines a $\mathcal B(\mathcal M)$-valued semispectral measure on $\partial \mathbb B^d.$ Moreover, for any $\alpha, \beta \in \mathbb Z^d_+$ and $x, y \in \mathcal M,$ we may infer from \eqref{quotient-norm} and the polarization identity that
\beqn
\inp{\phi(\alpha, \beta)x}{y}_{_\mathcal M} &\overset{\eqref{ip}}=& \inp{z^{\alpha} \otimes x}{z^{\beta} \otimes y} \\ 
%&=& \inp{ \mathscr M^\alpha_z (1 \otimes x)}{\mathscr M^{\beta}_z(1 \otimes y)} \\ 
&=& 
\inp{[z^\alpha  \otimes x]}{[z^{\beta} \otimes y]}_{\mathcal H} \\ &\overset{\eqref{action}}=& 
\inp{ S^\alpha ([1 \otimes x])}{S^{\beta}([1 \otimes y])}_{\mathcal H} \\  &\overset{\eqref{subn}}=& \int_{\partial \mathbb B^d}  \zeta^{\alpha} \overline{\zeta}^\beta d\inp{E(\zeta)Vx}{Vy}  \\ &=& \int_{\partial \mathbb B^d}  \zeta^{\alpha} \overline{\zeta}^\beta d\inp{F(\zeta)x}{y}. 
\eeqn
This completes the proof of the equivalence (i) $\Leftrightarrow$ (ii). The uniqueness part may be deduced from Riesz representation and Stone-Weierstrass approximation theorems.
\end{proof}

\vskip.2cm

{\bf Acknowledgments.} We convey our sincere thanks to Jan Stochel for drawing our attention to the operator approach to moment problems and also providing various references relevant to the complex moment problem discussed in Section 5.

{}


\begin{thebibliography}{}


\bibitem{Ag} J. Agler, A disconjugacy theorem for
Toeplitz operators, {\em Amer. J. Math.} {\bf 112} (1990), 1-14.
%
%\bibitem{AL} B. Abdullah and T. Le, The structure of $m$-isometric weighted shift operators,
%Oper.  Matrices, {\bf 10} (2016), 319-334.
%
%\bibitem{Ag-St} J. Agler and M. Stankus, $m$-isometric transformations of Hilbert spaces, I, II, III,
%{\it Integr. Equ. Oper.
%Theory} {\bf 21, 23, 24}
%(1995, 1995, 1996), 383-429, 1-48, 379-421.


\bibitem{AD} A. Aleksandrov and E. Doubtsov, 
Clark measures on the complex sphere,
{\it J. Funct. Anal.} {\bf 278} (2020), 108314, 30 pp.
 

%\bibitem{Al} A. Aleman, {\it The multiplication operators on Hilbert spaces of analytic functions},
%Habilitationsschrift, Fern-Universit$\ddot{\mbox{a}}$t Hagen, 1993.

%\bibitem{ARS} A. Aleman, S. Richter and C. Sundberg,  Beurling's theorem for the Bergman space, {\it Acta Math.} {\bf 177} (1996), 275-310.
%
%%\bibitem{AC} A. Anand and S. Chavan, A Moment Problem and Joint q-isometry Tuples, {\it Complex Analysis and Operator Theory}, {\bf 11} (2017), 785-810.
%
%%\bibitem{ACJS} A.  Anand,  S.  Chavan,  Z.  J.  Jab{\l}o\'{n}ski,  and  J.  Stochel,
%%A  solution  to  the  Cauchy  dual
%%subnormality problem for $2$-isometries, 2018, arXiv:1702.01264.

%\bibitem{ARSW}  N. Arcozzi, R. Rochberg, E. Sawyer and B. Wick, The Dirichlet space: a survey, {\it New York J. Math.} {\bf 17A} (2011), 45-86.

%\bibitem{AG}  D. Armitage and S. Gardiner, {\it Classical potential theory}, Springer Monographs in Mathematics. Springer-Verlag London, Ltd., London, 2001. xvi+333 pp.



\bibitem{At} A. Athavale, {On the intertwining of joint isometries}, \textit{J. Operator Theory}, \textbf{23} (1990), 339-350.

\bibitem{A-P} A. Athavale and S. Pedersen, Moment problems and subnormality, {\it J. Math. Anal. Appl.} {\bf 146} (1990), 434-441. 

\bibitem{ABR} S. Axler, P. Bourdon, and W. Ramey, {\it Harmonic function theory}, Second edition. Graduate Texts in Mathematics, 137. Springer-Verlag, New York, 2001. xii+259 pp.
%%\bibitem{CPT}  S. Chavan, D. K. Pradhan, and S. Trivedi, Multishifts on directed Cartesian products of rooted directed trees, {\it Dissertationes Math.} {\bf 527} (2017), 102 pp. 
%


\bibitem{BEKS}  M. Bhattacharjee, J. Eschmeier, D. Keshari, and J. Sarkar, Dilations, wandering subspaces, and inner functions, {\it Linear Algebra Appl.} {\bf 523} (2017), 263-280.

\bibitem{CS}  S. Chavan and V. Sholapurkar, Rigidity theorems for spherical hyperexpansions, {\it Complex Anal. Oper. Theory} {\bf 7} (2013), 1545-1568.

\bibitem{CK} S. Chavan and S. Kumar, Spherically balanced Hilbert spaces of formal power series in several variables. I, {\it J. Operator Theory} {\bf 72} (2014), 405-428.

\bibitem{CY}  S. Chavan and D. Yakubovich, Spherical tuples of Hilbert space operators, {\it Indiana Univ. Math. J.} {\bf 64} (2015), 577-612.

\bibitem{CSS}  D. Cicho\'{n}, J. Stochel and F. H. Szafraniec,  The complex moment problem: determinacy and extendibility, {\it Math. Scand.} {\bf 124} (2019), 263-288. 
 
%\bibitem{Co} J. Conway, {\it The Theory of Subnormal Operators}, Math. Surveys Monographs, {\bf 36}, Amer. Math. Soc. Providence, RI 1991.

%\bibitem{Cu}
%R. Curto, Applications of several complex variables to multiparameter spectral theory. Surveys of some recent results in operator theory, Vol. II, 25-90, {\it Pitman Res. Notes Math. Ser. {\bf 192}, Longman Sci. Tech.}, Harlow, 1988.


\bibitem{EKMR}  O. El-Fallah, K. Kellay, J. Mashreghi and T.  Ransford,  {\it A primer on the Dirichlet space}, Cambridge Tracts in Mathematics, 203. Cambridge University Press, Cambridge, 2014. xiv+211 pp.

 \bibitem{EL}  J. Eschmeier and S. Langend$\ddot{\mbox{o}}$rfer,  Multivariable Bergman shifts and Wold decompositions, {\it Inte. Equ. Oper. Th.} {\bf 90} (2018), Art. 56, 17 pp.

\bibitem{EP}  J. Eschmeier and M. Putinar, Some remarks on spherical isometries, Systems, approximation, singular integral operators, and related topics (Bordeaux, 2000), 271–291, {\it Oper. Theory Adv. Appl.} {\bf 129}, Birkh$\ddot{\mbox{a}}$user, Basel, 2001.



\bibitem{F}  B. Fuglede, The multidimensional moment problem, {\it Exposition. Math.} {\bf 1} (1983), 47-65.

\bibitem{GR}{J. Gleason and S. Richter}, $m$-Isometric commuting tuples of operators on a Hilbert space,
{\it Inte. Equ. Oper. Th.}  {\bf 56} (2006), 181-196.

%\bibitem{Kr} S. G. Krantz, {\it Function Theory of Several Complex Variables}
%(second edition), AMS Chelsea Publishing Series, vol. 340, Wadsworth $\&$ Brooks/Cole math. ser, AMS, Providence, R.I. 2001.

\bibitem{Gu}  C. Gu, Examples of m-isometric tuples of operators on a Hilbert space, {\it J. Korean Math. Soc.} {\bf 55} (2018), 225-251.

\bibitem{HM}  K. Hedayatian and A. Mohammadi-Moghaddam, Some properties of the spherical m-isometries, {\it J. Operator Theory} {\bf 79} (2018), 55-77.

\bibitem{JJS} Z.  J.  Jab{\l}o\'{n}ski,   I. B. Jung and  J.  Stochel, $m$-isometric operators and their local properties, 2019,  arXiv:1906.05215.

\bibitem{Le} T. Le, Decomposing algebraic $m$-isometric tuples, {\it J. Funct. Anal.}  {\bf 278} (2020), 108424, 14 pp.

\bibitem{O} A. Olofsson, 
A von Neumann-Wold decomposition of two-isometries,
{\it Acta Sci. Math. (Szeged)} {\bf 70} (2004), 715-726.

%\bibitem{P} V. Paulsen, {\it Completely Bounded Maps and Operator Algebras}, Cambridge University Press, 2002.

\bibitem{PR}  V. Paulsen and M. Raghupathi, {\it An introduction to the theory of reproducing kernel Hilbert spaces}, Cambridge Studies in Advanced Mathematics, 152. Cambridge University Press, Cambridge, 2016. x+182 pp.

\bibitem{Ri-2} S. Richter, Invariant subspaces of the Dirichlet shift, {\it J. Reine Angew. Math.} {\bf 386} (1988), 205-220.

\bibitem{Ri} S. Richter, 
A representation theorem for cyclic analytic two-isometries,
{\it Trans. Amer. Math. Soc.} {\bf 328} (1991), 325-349. 



\bibitem{Ri-3} S. Richter, A model for two-isometric operator tuples with finite defect,  
\url{https://www.math.utk.edu/richter/talk/Cornell_Sept_2011}, 2011.  

\bibitem{Ru0} W. Rudin,  
Pluriharmonic functions in balls,
{\it Proc. Amer. Math. Soc.} {\bf 62} (1977), 44-46. 

%\bibitem{Ru}  W. Rudin, {\it Functional analysis}, Second edition. International Series in Pure and Applied Mathematics. McGraw-Hill, Inc., New York, 1991. xviii+424 pp.

\bibitem{Ru-2}  W. Rudin, {\it Function theory in the unit ball of $\mathbb C^n,$} Reprint of the 1980 edition. Classics in Mathematics. Springer-Verlag, Berlin, 2008. xiv+436 pp.

\bibitem{Ry} E. Rydhe, 
Cyclic m-isometries and Dirichlet type spaces,
{\it J. Lond. Math. Soc.} {\bf 99} (2019), 733-756. 

\bibitem{Sh}
S. Shimorin, Wold-type decompositions and wandering subspaces for operators
close to isometries, {\it J. Reine Angew. Math.} {\bf 531} (2001), 147-189.

\bibitem{ST} J. Shohat and J. Tamarkin, 
{\it The Problem of Moments},
American Mathematical Society Mathematical surveys, vol. I. American Mathematical Society, New York, 1943. xiv+140 pp. 

\bibitem{Si-0} B. Simon, {\it Real analysis. A Comprehensive Course in Analysis}, Part 1. American Mathematical Society, Providence, RI, 2015. xviii+789 pp.

\bibitem{Si} B. Simon, {\it Harmonic analysis. A Comprehensive Course in Analysis}, Part 3. American Mathematical Society, Providence, RI, 2015. xviii+759 pp.

 

\bibitem{St} J. Stochel, 
Moment functions on real algebraic sets,
{\it Ark. Mat.} {\bf 30} (1992), 133-148. 

\bibitem{SS} J. Stochel and F. H. Szafraniec, 
Domination of unbounded operators and commutativity,
{\it J. Math. Soc. Japan} {\bf 55} (2003), 405-437.

%\bibitem{Sz} F. H. Szafraniec, Boundedness of the shift operator related to positive definite forms: an application to moment problems, {\it Ark. Mat.} {\bf 19} (1981), 251-259.

\bibitem{VW} A. Volberg and B. Wick, Bergman-type singular integral operators and the characterization of Carleson measures for Besov-Sobolev spaces and the complex ball, {\it Amer. J. Math.} {\bf 134} (2012), 949-992.
%%\bibitem{Si} B. Simon,
%%Real analysis. A Comprehensive Course in Analysis, Part 4. American Mathematical Society, Providence, RI, 2015.

%\bibitem{Z} K. Zhu, Spaces of holomorphic functions in the unit ball, Graduate Texts in Mathematics, 226. Springer-Verlag, New York, 2005. x+271 pp



\end{thebibliography}
\end{document}